\documentclass[a4paper,reqno,12pt]{amsart} 
\usepackage{amsmath} 

\usepackage{amssymb}      

\usepackage{yhmath}
\usepackage{mathdots}
\usepackage{MnSymbol}

\usepackage{amsthm}      

\usepackage{tikz,amsmath}
\usetikzlibrary{arrows}

\usepackage{epsfig}      

\usepackage{graphicx}
\usepackage[normalem]{ulem}
\usepackage[all,line,
]{xy} 
\CompileMatrices 

   \newcommand{\Aff}{{\operatorname{Aff}}}
 
   \newcommand{\Hom}{\operatorname{Hom}}

\newcommand{\Ad}{\operatorname{Ad}}

\newcommand{\id}{\operatorname{id}} 
\newcommand{\Aut}{\operatorname{Aut}}

\newcommand{\Span}{\operatorname{Span}}

 \newcommand{\supp}{\operatorname{supp}}

 \newcommand{\tr}{\operatorname{tr}}

\newcommand{\Br}{\operatorname{Br}}

   \theoremstyle{plain}
   \newtheorem{thm}{Theorem}[section]
   
   \newtheorem{lemma}[thm]{Lemma}  
   \newtheorem{cor}[thm]{Corollary}
   \theoremstyle{definition}

   \theoremstyle{remark}
   \newtheorem{obs}[thm]{Observation}
   \newtheorem{remark}[thm]{Remark}

\newtheorem{poem}[thm]{Additional properties}

\usepackage{mdframed,xcolor}

\definecolor{mybgcolor}{gray}{0.8}
\definecolor{myframecolor}{rgb}{.647,.129,.149}

\mdfdefinestyle{mystyle}{
  usetwoside=false,
  skipabove=0.6em plus 0.8em minus 0.2em,
  skipbelow=0.6em plus 0.8em minus 0.2em,
  innerleftmargin=.25em,
  innerrightmargin=0.25em,
  innertopmargin=0.25em,
  innerbottommargin=0.25em,
  leftmargin=-.75em,
  rightmargin=-0em,
  topline=false,
  rightline=false,
  bottomline=false,
  leftline=false,
  backgroundcolor=mybgcolor,
  splittopskip=0.75em,
  splitbottomskip=0.25em,
  innerleftmargin=0.5em,
  leftline=true,
  linecolor=myframecolor,
  linewidth=0.25em,
}

\newmdenv[style=mystyle]{important}

\usepackage{lipsum}

   \numberwithin{equation}{section}

        \date{\today}

\title[Possible temperatures]{The possible temperatures for flows on a simple AF algebra}
\author{Klaus Thomsen}

\DeclareMathOperator\coker{coker}

\date{\today}

\email{matkt@math.au.dk}
\address{Department of Mathematics, Aarhus University, Ny Munkegade, 8000 Aarhus C, Denmark}

\begin{document}

\maketitle

\section{Introduction}

A flow on a $C^*$-algebra $A$ is a continuous one-parameter group $\alpha = \left(\alpha_t\right)_{t \in \mathbb R}$ of automorphisms of $A$. In quantum statistical mechanics such a flow is sometimes interpreted as the time evolution of the observables of a physical system and the equilibrium states of the system are then given by states $\phi$ of $A$ that satisfy the trace-like condition 
$$
\phi(ab) = \phi(b\alpha_{i\beta}(a))
$$ 
for all $a,b \in A$ with $a$ analytic for $\alpha$, \cite{BR}. The number $\beta$ for which this holds is interpreted as the inverse temperature of the system in the state given by $\phi$, and $\phi$ is said to be a $\beta$-KMS state for $\alpha$. The $C^*$-algebra $A$ of observables in such a model is often a UHF algebra and it is therefore of interest to determine the KMS states and the possible inverse temperatures which can occur for flows on a UHF algebra. It follows from work by Powers and Sakai in the 70's, \cite{PS}, that when the flow is approximately inner the set of possible inverse temperatures is the whole real line $\mathbb R$, and Powers and Sakai conjectured that all flows on a UHF algebra are approximately inner. That this is not the case was proved by Matui and Sato in \cite{MS}, following work by Kishimoto on the AF algebra case, \cite{Ki2}. It is therefore an open question which sets of real numbers can occur as the set of inverse temperatures for a flow on a UHF algebra. The set is closed for any flow and since a UHF algebra has a unique trace state which is automatically a $0$-KMS state, it must contain zero. These are the only known restrictions. It is the purpose with this paper to show that it is also the only restrictions, at least when one only considers non-negative $\beta$, and not only for flows on UHF algebras, but in fact on all simple infinite dimensional unital AF algebras.

In the examples of Kishimoto, and also in the examples of Matui and Sato, the set of possible inverse temperatures is as small as it can be, consisting only of $0$, cf. Remark \ref{11-11-20} below. Here we will show that a variant of the method developed in \cite{Th5} can be used to construct a flow on any infinite dimensional simple unital AF algebra such that the set of inverse temperatures is any given lower bounded closed set of real numbers containing zero. For this we use a method which is a natural extension of ideas and methods from work by Bratteli, Elliott, Herman and Kishimoto in \cite{BEH}, \cite{BEK},\cite{Ki2}, and by Matui and Sato in \cite{MS}. As in these papers we depend on results from the classification of simple $C^*$-algebras. Besides the classification of AF algebras, which was fundamental already in \cite{BEH}, we apply the recent work by Castillejos, Evington, Tikuisis, White and Winter in \cite{CETWW}, where the classification results we use in turn are based on work by Elliott, Gong, Lin and Niu, \cite{GLN}, \cite{EGLN}, among others. However, when the algebra in question is UHF there is an alternative route to the desired result which we describe Remark \ref{14-12-20gx}.
 
The paper contains some material which appeared in \cite{Th5}; this is because I now consider \cite{Th5} a forerunner to the present paper and it will not be submitted for publication. The results of the present paper answer Question 2.6 in \cite{Th5} positively, but Question 6.2 in \cite{Th4} only partially. What remains is to determine if it is possible to remove the condition that the set is bounded below. This restriction is annoying and probably not necessary. However, in applications and also to many mathematicians, the (inverse) temperatures are non-negative and then the restriction to the lower bounded case is of no significance.

 In relation to this, one should observe the recent work by Christensen and Vaes, \cite{CV}, in which they connect the theory of KMS states on $C^*$-algebras to the theory of group actions on spaces. They also provide two examples of simple unital $C^*$-algebras on which there are flows with arbitrary KMS spectra. While it is certainly premature to suggest that this applies to all infinite dimensional unital simple $C^*$-algebras, it does seem appropriate to point out that we don't really know.

\smallskip

\emph{Acknowledgement} I am grateful to Y. Sato for comments on the first version of this paper which helped me navigate in the litterature on the classification of simple $C^*$-algebras, and I thank the referee of \cite{Th5} for his remarks. The work was supported by the DFF-Research Project 2 `Automorphisms and Invariants of Operator Algebras', no. 7014-00145B.

\section{Statement of results}

\begin{thm}\label{20-12-20c} Let $A$ be a unital infinite dimensional simple AF $C^*$-algebra. For each non-empty closed face $F$ in the tracial state space $T(A)$ of $A$ and for each closed and lower bounded set $K \subseteq \mathbb R$ of real numbers containing $0$ there is a $2\pi$-periodic flow $\alpha$ on $A$ such that
\begin{itemize}
\item there is a $\beta$-KMS state for $\alpha$ if and only if $\beta \in K$,
\item for $\beta \in K \backslash \{0\}$ the simplex of $\beta$-KMS states for $\alpha$ is affinely homeomorphic to $F$,
\item the simplex of $0$-KMS states for $\alpha$ is affinely homeomorphic to $T(A)$, and
\item the fixed point algebra $A^{\alpha}$ of $\alpha$ is an $AF$ algebra.
\end{itemize}
When $K$ is also upper bounded, and hence compact, $\alpha$ can be chosen such that the fixed point algebra $A^{\alpha}$ of $\alpha$ is a simple $AF$ algebra.
\end{thm}

When $K$ is not compact it is not possible to arrange that $A^{\alpha}$ is simple because when the set of inverse temperatures contains arbitrarily large numbers the flow must have ground states, and this is not possible for periodic flows whose fixed point algebra is simple, cf. Remark 3.5 in \cite{BEH}.

By taking $F$ to consist of a single extremal point in $T(A)$ we get the following

\begin{cor}\label{20-12-20d} Let $A$ be a unital infinite dimensional simple AF $C^*$-algebra. For each closed and lower bounded set $K \subseteq \mathbb R$ of real numbers containing $0$ there is a $2\pi$-periodic flow $\alpha$ on $A$ such that
\begin{itemize}
\item there is a $\beta$-KMS state for $\alpha$ if and only if $\beta \in K$,
\item for $\beta \in K \backslash \{0\}$ the $\beta$-KMS state is unique,
\item the simplex of $0$-KMS states for $\alpha$ is a affinely homeomorphic to $T(A)$, and
\item the fixed point algebra $A^{\alpha}$ of $\alpha$ is an $AF$ algebra.
\end{itemize}
When $K$ is also upper bounded, and hence compact, $\alpha$ can be chosen such that the fixed point algebra $A^{\alpha}$ of $\alpha$ is a simple $AF$ algebra.
\end{cor}



 

In the last section of the paper we combine Corollary \ref{20-12-20d} with methods and results from \cite{Th3} and \cite{Th4} to obtain the following additional corollaries. Recall that two compact Choquet simplexes are strongly affinely isomorphic when there is an affine bijection between them which restricts to a homeomorphism between the sets of extremal points. For Bauer simplexes this is the same as affine homeomorphism, but in general it is a weaker notion. See \cite{Th4}. 

\begin{cor}\label{15-12-20d} Let $U$ be a UHF algebra and let $K$ be a closed and lower bounded set of real numbers containing $0$. Let $\mathbb I$ be a finite or countably infinite collection of intervals in $\mathbb R$ such that $I = \mathbb R$ for at least one $I \in \mathbb I$. For each $I \in \mathbb I$ choose a compact metrizable Choquet simplex $S_I$ and for $\beta \in K$ set $\mathbb I_\beta = \left\{ I \in \mathbb I: \ \beta \in I\right\}$. There is a $2\pi$-periodic flow $\alpha$ on $U$ such that
\begin{itemize}
\item there is $\beta$-KMS state for $\alpha$ if and only if $\beta \in K$,
\item for each $I \in \mathbb I$ and $\beta \in (I\cap K) \backslash \{0\}$ there is a closed face $F_I$ in $S^{\alpha}_\beta$ strongly affinely isomorphic to $S_I$, and
\item for each $\beta \in K \backslash \{0\}$ and each $\omega \in S^\alpha_\beta$ there is a unique norm-convergent decomposition
$$
\omega = \sum_{I\in \mathbb I_\beta} \omega_I \ ,
$$
where $\omega_I \in \left\{ t \mu : \ t \in ]0,1], \ \mu \in F_I\right\}$. 
\end{itemize} 
\end{cor}

 \begin{cor}\label{27-10-20a}   Let $U$ be a UHF algebra and let $K$ be a closed and lower bounded set of real numbers containing $0$. There is a $2\pi$-periodic flow $\alpha$ on $U$ such that $S^\alpha_\beta = \emptyset$ if and only if $\beta \notin K$, and for $\beta, \beta' \in K$ the simplexes $S^{\alpha}_\beta$ and $S^\alpha_{\beta'}$ are not strongly affinely isomorphic unless $\beta = \beta'$. 
 \end{cor}

\section{Preparations}
In this paper all $C^*$-algebras are assumed to be separable and all traces and weights on a $C^*$-algebra are required to be non-zero, densely defined and lower semi-continuous. Concerning weights and in particular KMS weights we shall use notation and results from Sections 1.1 and 1.3 in \cite{KV}. Let $A$ be a $C^*$-algebra and $\theta$ a flow on $A$. Let $\beta \in \mathbb R$. A $\beta$-KMS weight for $\theta$ is a weight $\omega$ on $A$ such that $\omega \circ \theta_t = \omega$ for all $t$, and 
\begin{equation}\label{27-10-20c}
\omega(a^*a) \ = \ \omega\left(\theta_{-\frac{i\beta}{2}}(a) \theta_{-\frac{i\beta}{2}}(a)^*\right) \ \  \ \forall a \in D(\theta_{-\frac{i\beta}{2}}) \ .
\end{equation}
In particular, a $0$-KMS weight for $\theta$ is a $\theta$-invariant trace. It was shown by Kustermans in Theorem 6.36 of \cite{Ku} that this definition agrees with the one introduced by Combes in \cite{C}.
It is because of the formulation given by \eqref{27-10-20c}, which was not available when \cite{BEH} was written, that we are able to work with KMS weights throughout the present work. A bounded $\beta$-KMS weight is called a $\beta$-KMS functional and a $\beta$-KMS state when it is of norm $1$. 
 
 The first lemma can be considered as an updated version of a part of the discussion in Remark 3.3 of \cite{BEH}.

 \begin{lemma}\label{26-10-20} Let $B$ be a $C^*$-algebra and $\gamma \in \Aut(B)$ an automorphism of $B$. Let $\widehat{\gamma}$ be the dual action on $B \rtimes_{\gamma} \mathbb Z$ considered as a $2 \pi$-periodic flow. For $\beta \in \mathbb R$ the restriction map $\omega \mapsto \omega|_B$ is a bijection from the $\beta$-KMS weights for $\widehat{\gamma}$ onto the traces $\tau$ on $B$ with the property that $\tau \circ \gamma = e^{-\beta} \tau$. The inverse is the map $\tau \mapsto \tau \circ P$, where $P:  B \rtimes_{\gamma} \mathbb Z \to B$ is the canonical conditional expectation.
 \end{lemma}
 \begin{proof} Let $p_1 \leq p_2 \leq \cdots$ be an approximate unit in $B$ from the Pedersen ideal $K(B)$ of $B$ with the additional property that $p_k^2 \leq p^2_{k+1}$ for all $k$, cf. \cite{Pe}. Let $u$ be the canonical unitary multiplier of $B \rtimes_\gamma \mathbb Z$ such that $ubu^* = \gamma(b)$ for $b \in B$, and let $\tau$ be a trace on $B$ with the property that $\tau \circ \gamma = e^{-\beta} \tau$. For $a,b \in B, \ n,m \in \mathbb Z$ we have that
 $$
 P(p_kau^np_k^2bu^mp_k) = P(p_kbu^m p_k^2\widehat{\gamma}_{i\beta}(au^n)p_k)  = 0
 $$
 unless $m = -n$. Set $a' = p_ka\gamma^n(p_k), \ b' = p_kb\gamma^{-n}(p_k)$. Then $a',b' \in K(B)$, and by using Proposition 5.5.2 in \cite{Pe} we find that
 \begin{align*}
 & \tau\circ P(  p_kbu^{-n}p_k\widehat{\gamma}_{i \beta} (p_kau^np_k))  = e^{-n\beta}\tau(p_kbu^{-n}p_k^2au^np_k) \\
&  =e^{-n \beta} \tau( b'u^{-n} a'u^n) =\tau(\gamma^n(b')a') \\
& =  \tau(a'\gamma^n(b')) = \tau \circ P(p_kau^n p_k^2bu^{-n}p_k) \ ,  
  \end{align*} 
showing that $\tau \circ P$ is a $\beta$-KMS functional on $p_k(B \rtimes_\gamma \mathbb Z)p_k$ for the restriction of $\widehat{\gamma}$. We shall use repeatedly that when $b \geq 0$ in $B$ we have that
\begin{equation}\label{10-11-20a}
\lim_{k \to \infty} \tau(p_kbp_k) = \tau(b) \ ,
\end{equation} 
which follows from the lower semi-continuity of $\tau$ since $\tau(p_kbp_k) = \tau(\sqrt{b}p_k^2\sqrt{b})$ and $\sqrt{b}p_k^2\sqrt{b}$ increases to $b$ as $k \to \infty$. Let $x \in D(\widehat{\gamma}_{-\frac{i\beta}{2}})$. Then
\begin{align*}
& \tau\circ P\left(\widehat{\gamma}_{- \frac{i\beta}{2}}(x) \widehat{\gamma}_{- \frac{i\beta}{2}}(x)^*\right) = \lim_{k \to \infty} \tau\circ P\left(\widehat{\gamma}_{- \frac{i\beta}{2}}(x) p_k^2 \widehat{\gamma}_{- \frac{i\beta}{2}}(x)^*\right)\\
& = \lim_{k \to \infty}\lim_{l \to \infty} \tau \left(p_lP\left(\widehat{\gamma}_{- \frac{i\beta}{2}}(x) p_k^2 \widehat{\gamma}_{- \frac{i\beta}{2}}(x)^*\right)p_l\right) \ . 
\end{align*}
We have shown above that $\tau \circ P$ is a $\beta$-KMS functional on $p_l(B \rtimes_\gamma \mathbb Z)p_l$ and when $l \geq k$ this gives 
\begin{align*}
&\tau\left(p_l P\left(\widehat{\gamma}_{- \frac{i\beta}{2}}(x) p_k^2 \widehat{\gamma}_{- \frac{i\beta}{2}}(x)^*\right)p_l\right)  = \tau\circ P\left(\widehat{\gamma}_{- \frac{i\beta}{2}}(p_lxp_k)  \widehat{\gamma}_{- \frac{i\beta}{2}}(p_lxp_k)^*\right) \\
& = \ \tau \circ P((p_lxp_k)^*p_lxp_k) = \tau \circ P(p_kx^*p_l^2xp_k)  \leq  \tau\left(p_kP(x^*x)p_k\right) \ .
\end{align*}
By using \eqref{10-11-20a} we conclude that
$$
\tau \circ P\left(\widehat{\gamma}_{- \frac{i\beta}{2}}(x) \widehat{\gamma}_{- \frac{i\beta}{2}}(x)^*\right) \ \leq \ \tau \circ P(x^*x) \ .
$$
Similarly,
\begin{align*}
& \tau \circ P(x^*x) = \lim_{k \to \infty} \tau \circ P(x^*p_k^2x) = \lim_{k \to \infty} \lim_{l \to \infty} \tau \circ P(p_lx^*p_k^2 x p_l) \ .
\end{align*}
When $l \geq k$,
\begin{align*}
&\tau \circ P(p_lx^*p_k^2 xp_l) = \tau \circ P( \widehat{\gamma}_{-\frac{i\beta}{2}}(p_k xp_l) \widehat{\gamma}_{-\frac{i\beta}{2}}(p_kxp_l)^*) \\
& = \tau \circ P(p_k \widehat{\gamma}_{-\frac{i\beta}{2}}( x)p_l^2 \widehat{\gamma}_{-\frac{i\beta}{2}}(x)^*p_k \  \leq \tau\left(p_kP( \widehat{\gamma}_{-\frac{i\beta}{2}}( x) \widehat{\gamma}_{-\frac{i\beta}{2}}(x)^*)p_k\right) ,
\end{align*}
and we find therefore that
\begin{align*}
&\tau \circ P(x^*x) \leq   \lim_{k \to \infty}  \tau \left(p_kP( \widehat{\gamma}_{-\frac{i\beta}{2}}( x) \widehat{\gamma}_{-\frac{i\beta}{2}}(x)^*)p_k\right) =  \tau \circ P( \widehat{\gamma}_{-\frac{i\beta}{2}}( x) \widehat{\gamma}_{-\frac{i\beta}{2}}(x)^*) \ .
\end{align*}
We conclude therefore first that $\tau \circ P(x^*x) = \tau \circ P( \widehat{\gamma}_{-\frac{i\beta}{2}}( x) \widehat{\gamma}_{-\frac{i\beta}{2}}(x)^*)$, and then that $\tau \circ P$ is a $\beta$-KMS weight.

Let $\omega$ be a $\beta$-KMS weight for $\widehat{\gamma}$. Then $\omega(b^*b) = \omega(bb^*)$ for all $b \in B$ because $\widehat{\gamma}_{-i\frac{\beta}{2}}(b) = b$. Using Riemann sum approximations to the integral
$$
P(a) = (2\pi)^{-1} \int_0^{2\pi} \widehat{\gamma}_t(a) \ \mathrm d t \ ,
$$
it follows from the $\widehat{\gamma}$-invariance and lower semi-continuity of $\omega$ that
\begin{equation}\label{26-10-20b}
\omega \circ P(a) \leq \omega(a)
\end{equation}
 when $ a \geq 0$ in $B\rtimes_\gamma \mathbb Z$. In particular, $\omega|_B$ is densely defined since $\omega$ is, and we conclude that $\tau = \omega|_B$ is a trace on $B$.  Let $a \geq 0$. Then 
$$
\omega \circ \gamma(a) = \omega(uau^*) = e^{-\beta}\omega\left( \widehat{\gamma}_{-\frac{i\beta}{2}}(u\sqrt{a}) \widehat{\gamma}_{-\frac{i\beta}{2}}(u\sqrt{a})^*\right) = e^{-\beta} \omega(a) \ .
$$
It follows that $\tau \circ \gamma = e^{-\beta}\tau$. It remains now only to show that $\omega = \omega \circ P$. Note that because $\omega \circ \widehat{\gamma}_t = \omega$ we find for all $b\in B, \ n \in \mathbb Z$, that
$$
\omega(p_kbu^np_k) = \begin{cases} 0, & \ n \neq 0 \\\omega(p_kbp_k) , & \ n = 0 \end{cases} \ = \ \omega(p_kP(bu^n)p_k) \ ,
$$
 implying that $\omega(p_k \ \cdot \ p_k) = \omega(p_kP( \ \cdot \ )p_k)$. Let $a \in B \rtimes_\gamma \mathbb Z$, $ a \geq 0$. Since $\omega|_B$ is a trace we can use \eqref{10-11-20a} to conclude that
$$
\lim_{k \to \infty} \omega(p_kap_k) = \lim_{k \to \infty} \omega(p_kP(a)p_k) =  \ \omega(P(a)) \ .
$$
Since $\lim_{k \to \infty} p_kap_k = a$ the lower semi-continuity of $\omega$ implies now that $\omega(a) \leq \omega(P(a))$. Combined with \eqref{26-10-20b} this yields the desired conclusion that $\omega \circ P = \omega$.
\end{proof}

\begin{remark}\label{11-11-20} It follows from Lemma \ref{26-10-20} that when $B$ has a unique trace the dual action will have no $\beta$-KMS weights for $\beta \neq 0$, and if in addition $B \rtimes_{\gamma} \mathbb Z$ is simple the restriction of the dual action to any corner $e(B \rtimes_\gamma \mathbb Z)e$ given by a $\widehat{\gamma}$-invariant non-zero projection $e$ will have no $\beta$-KMS states for $\beta \neq 0$ by Remark 3.3 in \cite{LN} or Theorem 2.4 in \cite{Th2}. This observation applies to the flows that were shown not to be approximately inner in \cite{Ki2} and \cite{MS}.
\end{remark}

Taking $\beta =0$ in Lemma \ref{26-10-20} we get a special case which we shall need. It is probably known.
 
  \begin{cor}\label{28-10-20a} Let $B$ be a $C^*$-algebra and $\gamma \in \Aut(B)$ an automorphism of $B$.  The restriction map $\omega \mapsto \omega|_B$ is a bijection from the $\widehat{\gamma}$-invariant traces $\omega$ on $B \rtimes_\gamma \mathbb Z$ onto the $\gamma$-invariant traces on $B$. The inverse is the map $\tau \mapsto \tau \circ P$, where $P:  B \rtimes_{\gamma} \mathbb Z \to B$ is the canonical conditional expectation.
\end{cor}  
  
We remark that in general $B \rtimes_\gamma \mathbb Z$ can have traces that are not invariant under the dual action $\widehat{\gamma}$ and hence are not determined by their restriction to $B$; also in cases where $B$ is UHF and $B \rtimes_\gamma \mathbb Z$ is simple. For the present purposes it is crucial that we can circumvent this issue thanks to the following lemma which is suggested by the work of Matui and Sato in \cite{MS}. In fact, most of it appears implicitly in \cite{MS}.

Recall that a $C^*$-algebra $A$ is stable when $A \otimes \mathbb K \simeq A$, where $\mathbb K$ is the $C^*$-algebra of compact operators on an infinite dimensional separable Hilbert space. Recall also that a partially ordered group $(G,G^+)$ has large denominators when the following holds: For any $a \in G^+$ and any $n \in \mathbb N$ there is an element $b \in G$ and an $m \in \mathbb N$ such that $ nb\leq a \leq mb$, cf. \cite{Ni}.

\begin{lemma}\label{28-10-20} Let $B$ be a stable AF-algebra and $\gamma \in \Aut(B)$. Assume that $K_0(B)$ has large denominators. There is an automorphism $\gamma' \in \Aut(B)$ such that
\begin{enumerate}
\item[a)] $\gamma'_* = \gamma_*$ on $K_0(B)$,
\item[b)] the restriction map 
$\mu \ \mapsto \ \mu|_B$
is a bijection from traces $\mu $ on $B \rtimes_{\gamma'} \mathbb Z$ onto the $\gamma'$-invariant traces on $B$, 
\item[c)] $B \rtimes_{\gamma'} \mathbb Z$ is $\mathcal Z$-stable; i.e $(B \rtimes_{\gamma'} \mathbb Z)\otimes \mathcal Z \simeq B \rtimes_{\gamma'} \mathbb Z$ where $\mathcal Z$ denotes the Jiang-Su algebra, \cite{JS}, and
\item[d)] $B \rtimes_{\gamma'} \mathbb Z$ is stable.
\end{enumerate} 
\end{lemma}
\begin{proof} The first step is to show that there is an isomorphism $\phi : B \to B \otimes \mathcal Z$ such that $\phi_* = {(\id_B \otimes 1_{\mathcal Z})}_*$. Since $\mathcal Z$ is strongly self-absorbing in the sense of Toms and Winter, \cite{TW}, it suffices for this to show that $B$ is $\mathcal Z$-stable. By Corollary 3.4 in \cite{TW} it is enough to show that $eBe$ is $\mathcal Z$-stable for any non-zero projection $e \in B$, and hence by Theorem 2.3 in \cite{TW} it is enough to show that $eBe$ is approximately divisible. Now $K_0(eBe)$ has large denominators since we assume that $K_0(B)$ has, and it follows therefore from Lemma 4.4 in \cite{Th1} that for any Bratteli diagram for $eBe$ the multiplicity matrices describing the embedding of a given level into the following levels will have the property that the minimum of the non-zero entries increases to infinity. This implies that $eBe$ is approximately divisible, cf. \cite{BKR}. We deduce in this way the existence of $\phi$.

For the second step we use \cite{Sa} to obtain an automorphism $\theta$ of $\mathcal Z$ with the weak Rohlin property; that is, with the property that for each $k \in \mathbb N$ there is a sequence $\{f_n\}$ in $\mathcal Z$ such that
\begin{itemize}
\item $0 \leq f_n \leq 1_{\mathcal Z}$ for all $n$,
\item $\lim_{n \to \infty} f_na -af_n = 0$ for all $a \in \mathcal Z$,
\item $\lim_{n \to \infty} \theta^j(f_n)f_n = 0$ for $j=1,2,3,\cdots, k$, and
\item $\lim_{n \to \infty} \tau\left(1_{\mathcal Z} - \sum_{j=0}^{k}\theta^j(f_n)\right) = 0$ ,
\end{itemize}
where $\tau$ is the trace state of $\mathcal Z$. From the first step, applied twice, it follows that there is an isomorphism $\phi : B \to B \otimes \mathcal Z \otimes \mathcal Z$ such that $\phi_* = (\id_B \otimes 1_{\mathcal Z} \otimes 1_{\mathcal Z})_*$. Since $B$ is stable there is also a $*$-isomorphism $\psi_0 : B \to \mathbb K \otimes B$ such that ${\psi_0}_* = (e \otimes \id_B)_*$, where $e$ is a minimal non-zero projection in $\mathbb K$. Set 
$$
\psi = \left(\psi_0 \otimes \id_{\mathcal Z} \otimes  \id_{\mathcal Z}\right) \circ \phi : B \to \mathbb K \otimes B \otimes \mathcal Z \otimes \mathcal Z \ 
$$
and
$$
\gamma' = \psi^{-1} \circ \left(\id_{\mathbb K} \otimes \gamma \otimes \theta \otimes \id_{\mathcal Z}\right) \circ \psi \ .
$$ 
It follows from the K\"unneth theorem that $\psi_* = (e\otimes \id_B \otimes 1_{\mathcal Z} \otimes 1_{\mathcal Z})_*$ and hence $\left(\id_{\mathbb K} \otimes \gamma \otimes \theta \otimes \id_{\mathcal Z}\right)_* \circ \psi_* = (e \otimes \gamma \otimes 1_{\mathcal Z} \otimes 1_{\mathcal Z})_*$. Thus $\gamma'_* = \gamma$.
To complete the proof it suffices to verify that b), c) and d) hold when $B$ is replaced by $\mathbb K \otimes B \otimes \mathcal Z\otimes \mathcal Z$ and $\gamma'$ by $\id_{\mathbb K} \otimes \gamma \otimes \theta \otimes \id_{\mathcal Z}$. The properties c) and d) follow immediately because $ \mathcal Z \otimes \mathcal Z \simeq \mathcal Z$, $\mathbb K \otimes \mathbb K \simeq \mathbb K$ and
$$
(\mathbb K \otimes B \otimes \mathcal Z \otimes \mathcal Z) \rtimes_{\id_{\mathbb K} \otimes\gamma \otimes \theta \otimes \id_{\mathcal Z}} \mathbb Z \ \simeq \ \mathbb K \otimes \left((B \otimes \mathcal Z)  \rtimes_{\gamma \otimes \theta} \mathbb Z \right) \otimes \mathcal Z \ .
$$
In the final step we verify b). For this, set $B' = \mathbb K \otimes B \otimes \mathcal Z$ and $\gamma'' = \id_{\mathbb K} \otimes \gamma \otimes \id_{\mathcal Z}$. Then 
$$
( \mathbb K \otimes B \otimes \mathcal Z \otimes \mathcal Z) \rtimes_{\id_{\mathbb K} \otimes \gamma \otimes \theta \otimes \id_{\mathcal Z}} \mathbb Z \ \simeq \ (B'\otimes \mathcal Z) \rtimes_{\gamma'' \otimes \theta} \mathbb Z \ ,
$$
and it suffices to verify that b) holds when $B$ is replaced by $B' \otimes \mathcal Z$ and $\gamma'$ by $\gamma''\otimes \theta$. Let $P : (B' \otimes \mathcal Z) \rtimes_{\gamma''\otimes \theta} \mathbb Z \to B' \otimes \mathcal Z$ be the canonical conditional expectation. In view of Corollary \ref{28-10-20a} what remains is to consider a trace $\mu$ on $(B '\otimes \mathcal Z) \rtimes_{\gamma''\otimes \theta} \mathbb Z$ and show that $\mu = \mu \circ P$. For this, let $\{q_n\}$ be an approximate unit in $B'$ consisting of projections. Since $q_n \otimes 1_{\mathcal Z}$ is a projection and therefore contained in the Pedersen ideal, $\mu(q_n \otimes 1_{\mathcal Z}) < \infty$, cf. \cite{Pe}. When $a \geq 0$ in $ (B' \otimes \mathcal Z) \rtimes_{\gamma''\otimes \theta} \mathbb Z $, 
$$
\mu(a) = \lim_n \mu\left(\sqrt{a}(q_n \otimes 1_{\mathcal Z}) \sqrt{a}\right) =  \lim_n \mu\left((q_n \otimes 1_{\mathcal Z})a(q_n\otimes 1_{\mathcal Z})\right) \ ,
$$
and similarly $
\mu(P(a)) =   \lim_n \mu\left((q_n \otimes 1_{\mathcal Z})P(a)(q_n\otimes 1_{\mathcal Z})\right)$.
It suffices therefore to show that 
$$
\mu\left((q_n \otimes 1_{\mathcal Z})a(q_n\otimes 1_{\mathcal Z})\right) = \mu\left((q_n \otimes 1_{\mathcal Z})P(a)(q_n\otimes 1_{\mathcal Z})\right)
$$ 
for all $n$. For this note that $x \mapsto \mu\left((q_n \otimes 1_{\mathcal Z})x(q_n\otimes 1_{\mathcal Z})\right)$ extends to a bounded positive linear functional of norm $\mu\left(q_n \otimes 1_{\mathcal Z}\right)$ on $(B' \otimes \mathcal Z) \rtimes_{\gamma''\otimes \theta} \mathbb Z$. It suffices therefore to consider positive elements $b \in B'$, $z \in \mathcal Z$ and $k \in \mathbb Z \backslash \{0\}$, and show that 
\begin{equation}\label{23-10-20}
\mu((q_n \otimes 1_{\mathcal Z})(b\otimes z) u^k(q_n\otimes 1_{\mathcal Z})) = 0
\end{equation}
when $u$ is the canonical unitary in the multiplier algebra of $(B' \otimes \mathcal Z) \rtimes_{\gamma''\otimes \theta} \mathbb Z$ such that $\Ad u = \gamma'' \otimes \theta$ on $B' \otimes \mathcal Z$. Since the complex conjugate of $\mu((q_n \otimes 1_{\mathcal Z})(b\otimes z) u^{-k}(q_n\otimes 1_{\mathcal Z}))$ is $\mu((q_n \otimes 1_{\mathcal Z})({\gamma''}^k(b)\otimes \theta^k(z)) u^{k}(q_n\otimes 1_{\mathcal Z}))$
we may assume $k \geq 1$. Let $\epsilon > 0$. Since 
$a \mapsto \mu(q_n \otimes a)$ is a bounded trace on $\mathcal Z$ it must be a scalar multiple of the unique trace state $\tau$ of $\mathcal Z$. It follows therefore from the properties of $\theta$ that there are elements $0 \leq g_j \leq 1_{\mathcal Z}, \ j = 0,1,2, \cdots , k$, in $\mathcal Z$ such that 
\begin{align}
&\left\|g_jz-zg_j\right\| \leq \epsilon \ \text{for all} \ j \ ,\label{23-10-20a} \\
& \left\|g_ig_j \right\| \leq \epsilon \ \text{for all} \ i,j, \ i \neq j , \label{23-10-20b} \\
& \left\|\theta^k(g_j)g_j\right\| \leq \epsilon \ \text{for all } \ j, \ \text{and} \label{23-10-20c} \\
&\left|\mu(q_n \otimes (1_{\mathcal Z}- \sum_{j=0}^{k} g_j)) \right| \ \leq \ \epsilon \mu(q_n\otimes 1_{\mathcal Z}). \ \label{23-10-20d}
\end{align}
In the following, when $s$ and $t$ are complex numbers such that for all $\delta > 0$ the $\epsilon$ in \eqref{23-10-20a}-\eqref{23-10-20d} can be chosen so small that $|s-t|\leq  \delta$, we will write $s \sim t$. Set
$$
y = 1_{\mathcal Z} - \sum_{j=0}^k g_j \ .
$$
It follows from \eqref{23-10-20a} that
$$
\mu((q_n \otimes 1_{\mathcal Z})(b\otimes z y) u^k(q_n\otimes 1_{\mathcal Z})) \ \sim \ \mu((q_n \otimes 1_{\mathcal Z})(b\otimes yz ) u^k(q_n\otimes 1_{\mathcal Z}))\ ,
$$
and from the Cauchy-Schwarz inequality that
\begin{equation}\label{23-10-20e}
\begin{split}
&\left|\mu((q_n \otimes 1_{\mathcal Z})(b\otimes yz ) u^k(q_n\otimes 1_{\mathcal Z}))\right| \ =  \ \left|\mu((q_n \otimes y)(b\otimes z ) u^k(q_n\otimes 1_{\mathcal Z}))\right| \\
& \leq  \ \sqrt{  \mu(q_n \otimes y^2)}\|b\|\|z\| \sqrt{\mu(q_n\otimes 1_{\mathcal Z})} \ .
\end{split}
\end{equation}
Note that it follows from \eqref{23-10-20b} that $\left\| (\sum_{j=0}^{k} g_j)^2 - \sum_{j=0}^{k} g_j^2\right\| \leq (k+1)^2\epsilon$ and hence
$$(\sum_{j=0}^{k} g_j)^2 \leq \sum_{j=0}^{k} g_j^2 + (k+1)^2 \epsilon 1_{\mathcal Z}\ \leq \  \sum_{j=0}^{k} g_j + (k+1)^2 \epsilon 1_{\mathcal Z}\ .
$$ 
It follows that $y^2 \leq y + (k+1)^2 \epsilon 1_{\mathcal Z}$ and then from \eqref{23-10-20d} that
$$
0 \ \leq  \ \mu(q_n \otimes y^2) \ \leq  \ ((k+1)^2 +1) \mu(q_n\otimes 1_{\mathcal Z}) \epsilon \ .
$$
Combined with \eqref{23-10-20e} this shows that $\mu((q_n \otimes 1_{\mathcal Z})(b\otimes yz ) u^k(q_n\otimes 1_{\mathcal Z})) \sim 0$, so to obtain \eqref{23-10-20} it suffices to show that $\mu((q_n \otimes 1_{\mathcal Z})(b\otimes zg_j) ) u^k(q_n\otimes 1_{\mathcal Z})) \sim 0$ for each $j$. It follows from \eqref{23-10-20a} that
$$
\mu((q_n \otimes 1_{\mathcal Z} )(b\otimes zg_j) ) u^k(q_n\otimes 1_{\mathcal Z})) \ \sim \ \mu((q_n \otimes 1_{\mathcal Z})(b\otimes \sqrt{g_j} z\sqrt{g_j}) ) u^k(q_n\otimes 1_{\mathcal Z})) \ . 
$$ 
Using Proposition 5.5.2 in \cite{Pe} for the second equality we get that 
\begin{align*}
&\mu((q_n \otimes 1_{\mathcal Z})(b\otimes \sqrt{g_j} z\sqrt{g_j})  u^k(q_n\otimes 1_{\mathcal Z})) \\
&= \mu((q_n \otimes \sqrt{g_j})(b\otimes  z \sqrt{g_j} ) u^k(q_n\otimes 1_{\mathcal Z})) \\
& = \mu((q_n \otimes 1_{\mathcal Z})(b\otimes  z \sqrt{g_j} ) u^k(q_n\otimes \sqrt{g_j})) \\
& = \mu((q_n \otimes 1_{\mathcal Z})(b{\gamma''}^k(q_n)\otimes  z \sqrt{g_j} \theta^k(\sqrt{g_j}) ) u^k(q_n\otimes 1_{\mathcal Z})) \ .
\end{align*}
It follows therefore from \eqref{23-10-20c} that
$ \mu((q_n \otimes 1_{\mathcal Z})(b\otimes zg_j) ) u^k(q_n\otimes 1_{\mathcal Z})) \sim 0$ as desired.
\end{proof}

The next lemma is well-known.

\begin{lemma}\label{23-10-20g} Let $B$ be an AF algebra. There is a bijective correspondence between traces $\tau$ on $B$ and the set of non-zero positive homomorphisms $ K_0(B) \to \mathbb R$. The bijection is given by the formula $\tau_*[e] = \tau(e)$ when $e$ is a projection in $B$.
\end{lemma} 

\section{Proof of Theorem \ref{20-12-20c} when $K$ is compact}\label{compact}

The case where $K$ is compact is sufficiently different from the case where this is not the case, that we divide the proof accordingly. However, the general setup which we now describe will be the same in the two cases.

\subsection{Setting the stage}\label{stage}

Let $A$ be an infinite dimensional simple unital AF algebra. The tracial state space $T(A)$ of $A$ is a compact metrizable Choquet simplex which we denote by $\Delta$. We shall handle $\Delta$ as an abstract simplex and when an element $x \in \Delta$ is considered as a trace on $A$ we denote it by $\tr_x$. The $K_0$-group $K_0(A)$ is a simple non-cyclic dimension group and we denote by $K_0(A)^+$ the semi-group of positive elements of $K_0(A)$. The unit $1$ in $A$ represents an element $[1] \in K_0(A)$ which is an order unit in $K_0(A)$. To simplify notation we set
$$
(K_0(A),K_0(A)^+,[1]) \ = \ (H,H^+, u) \ .
$$
Each element $x \in \Delta$ defines in a canonical way a homomorphism ${\tr_x}_* : H \to \mathbb R$, and since $A$ is AF the state space 
$$
 \left\{\phi \in \Hom(H,\mathbb R) : \ \phi(H^+) \subseteq [0,\infty) , \ \phi(u) = 1 \right\}
$$
of $(H,H^+,u)$ is affinely homeomorphic to $\Delta$ via the map $\Delta \ni x \mapsto {\tr_x}_*$. Let $\Aff(\Delta)$ denote the space of real-valued continuous affine functions on $\Delta$. There is then a homomorphism
$$
\theta : H \to \Aff(\Delta)
$$
defined such that $\theta(h)(x) = {\tr_x}_*(h)$. By Theorem 4.11 in \cite{GH},
\begin{itemize}
\item $\theta(H)$ is norm-dense in $\Aff(\Delta)$, and
\item $H^+ \ = \ \left\{h \in H: \ \theta(h)(x) > 0 \ \ \forall x \in \Delta \right\} \cup \{0\}$ .
\end{itemize}
 Set
$$
G \ = \ \oplus_{\mathbb Z} H \ ,
$$  
and define an automorphism $\rho$ of $G$ such that
\begin{equation}\label{rho}
\rho\left((g_n)_{n \in \mathbb Z}\right) \ = \left(g_{n+1}\right)_{n \in \mathbb Z} \  .
\end{equation}
The main difference between the two cases, compact and not compact, is in the definition of the ordering in $G$, and although many arguments are the same in the two cases, and will only be given once, we make a clear distinction between the two cases in order to increase the readability.

\subsection{The first construction}\label{first}

We assume now that the set $K$ from Theorem \ref{20-12-20c} is compact and set $L = \left\{e^{\beta} : \ \beta \in K\right\}$. Let $A(\Delta \times L)$ denote the real vector space of functions $f \in C_{\mathbb R}(\Delta \times L)$ for which the map $\Delta \ni x \mapsto f(x,t)$ is affine for all $t \in L$; clearly a norm-closed subspace of $C_{\mathbb R}(\Delta \times L)$. Define 
$\Sigma : G \to A(\Delta \times L)$
such that
$$
\Sigma\left(g\right)(x,t) \ = \ \sum_{m \in \mathbb Z} \theta(g_m)(x)t^{m}  \ .
$$
Let $F$ be a non-empty closed face in $\Delta$ and set
$$
G^+ \ = \ \left\{ g \in G : \ \Sigma(g)(x,t) \ > \ 0 \ \ \forall (x,t) \in (F \times L) \cup (\Delta \times \{1\}) \right\} \cup \{0\} \  ;
$$
a semi-group in $G$ which turns $G$ into a partially ordered abelian group. We aim to show that $G$ is a simple dimension group.

For $(k,l) \in \mathbb Z^2$ let $g_{k,l} \in C_{\mathbb R}(L)$ be the function 
$$
g_{k,l}(t) \ = \ t^k-t^l \ .
$$

\begin{lemma}\label{15-12-20xx} $\Span \{g_{k,l} : \ (k,l) \in \mathbb Z^2 \}$ is dense in $\left\{ h \in C_{\mathbb R}(L): \ h(1) = 0 \right\}$.
\end{lemma}
\begin{proof} Note that $tg_{k,l}(t) = g_{k+1,l+1}(t)$. It follows that $\Span \{g_{k,l} : \ (k,l) \in \mathbb Z^2 \}$ contains all functions of the form $P(t)\left(t - t^{-1}\right)$ where $P$ is a polynomial. It follows therefore from Weierstrass' theorem that the closure of $\Span \{g_{k,l} : \ (k,l) \in \mathbb Z^2 \}$ contains all functions in $C_{\mathbb R}(L)(t -t^{-1})$. Let $ h \in C_{\mathbb R}(L)$ satisfy that $ h(1) = 0$, and let $\epsilon > 0$. There is a function $h_1 \in C_{\mathbb R}(L)$ such that $h_1(t) = 0$ for all $t$ in a neighborhood of $1$ and $\sup_{t \in L}|h(t)-h_1(t)| \leq \epsilon$. Since 
$$
h_1(t) = \frac{h_1(t)}{t -t^{-1}} (t-t^{-1})
$$ 
it follows first that $h_1$ is in the closure of $\Span \{g_{k,l} : \ (k,l) \in \mathbb Z^2 \}$, and then that so is $h$.
\end{proof}

When $h \in H$ and $k,l \in \mathbb Z, \ k \neq l$, we denote in the following by $[[h]]^{k,l}$ the element of $G$ such that
$$
\left( [[h]]^{k,l}\right)_i \ = \ \begin{cases} h, & \ i = k \\ -h , & \ i= l\\ 0, & \ i \notin \{k,l\} \ , \end{cases} 
$$
and by $[[h]]\in G$ the element with $[[h]]_0 = h$ and $[[h]]_i = 0$ when $i \neq 0$.

\begin{lemma}\label{16-12-20xx} $\Sigma(G)$ is dense in $A(\Delta \times L)$.
\end{lemma}
\begin{proof} Note that $\Sigma\left([[h]]^{k,l}\right)(x,t) \ = \ \theta(h)(x)g_{k,l}(t)$ and $\Sigma([[h]])(x,t) = \theta(h)(x)$. Hence $\Sigma(G)$ contains all functions of the form 
$$
\theta(h)(x)(t^k-t^l) \ ,
$$ 
where $k,l \in \mathbb Z$, as well as all functions of the form $\theta(h)(x)$ when $h \in H$. Since $\theta(H)$ is dense in $\Aff(\Delta)$ it follows from Lemma \ref{15-12-20xx} that the closure $\overline{\Sigma(G)}$ of $\Sigma(G)$ contains all functions of the form $(x,t) \mapsto a(x)f(t)$ where $a \in \Aff(\Delta)$ and $f \in C_{\mathbb R}(L)$. A wellknown partition of unity argument shows then that $A(\Delta \times L) = \overline{\Sigma(G)}$.
\end{proof}

\begin{lemma}\label{21-12-20} Let $a \in \Aff(F)$. For each $x_0 \in \Delta \backslash F$ there is a function $b \in \Aff(\Delta)$ such that $b|_F = a$ and $b(x_0) > a(y)$ for all $y \in F$.
\end{lemma}
\begin{proof} Choose $u,v \in \mathbb R$ such that $u \leq a(y) \leq v$ for all $y \in F$. By the Hahn-Banach separation theorem there is a $c \in \Aff (\Delta)$ such that $c(y) \leq 0$ for all $y \in F$ and $c(x_0)> 0$. Multiplying by a positive number we can arrange that $c(x_0) >  v -u $. Set $C  = \sup_{x \in \Delta} c(x)$. Then $c(y) \leq a(y) - u \leq v-u + C$ for all $y \in F$ and $c(x) \leq v-u+C$ for all $x \in \Delta$. It follows from Edwards separation theorem, Theorem 3 in \cite{Ed} or Corollary 7.7 in \cite{AE}, that there is a $b' \in \Aff(\Delta)$ such that $b'|_F = a-u$ and $c \leq b' \leq v-u +C$ on $\Delta$. Set $b = b' +u$. Then $b|_F =a$ and $b(x_0) \geq c(x_0) + u > v$.
\end{proof}

\begin{lemma}\label{18-12-20} $G$ is a simple dimension group.
\end{lemma}
\begin{proof} It is easy to see that $G$ is unperforated and that $G^+ \cap (-G^+) = \{0\}$. Since $(F \times L) \cup (\Delta \times \{1\})$ is compact it is also clear that every non-zero element of $G^+$ is an order unit for $G$, so what remains is to show that $G$ has the Riesz interpolation property. Let $c^i,d^i, \ i \in \{1,2\}$, be elements of $G$ such that $c^i \leq d^j$ for all $i,j \in \{1,2\}$. If $c^{i'} = d^{j'}$ for some $i',j'$, set $z = c^{i'}$. Then $c^i \leq z \leq d^j$ for all $i,j$. Assume instead that $c^i < d^j$ for all $i,j$. Then $\Sigma(c^i)(x,t)   < \Sigma(d^j)(x,t)$ for all $i,j$ and all $(x,t) \in (F\times L) \cup (\Delta \times \{1\})$. Define $c,d \in C_{\mathbb R}(\Delta \times L)$ such that
$$
c(x,t) \ = \ \max \{ \Sigma(c^1)(x,t),\Sigma(c^2)(x,t) \}
$$
and
$$
d(x,t) \ = \   \min \{ \Sigma(d^1)(x,t),\Sigma(d^2)(x,t) \} \ .
$$
Let $\delta > 0$. If $\delta$ is small enough
$$
c(x,t) + \delta \ < \ d(x,t) - \delta
$$
for all $(x,t) \in (F\times L) \cup (\Delta \times \{1\})$. Fix $t \in L \backslash\{1\}$. As a closed face in the simplex $\Delta$ the set $F$ is itself a Choquet simplex and hence the space $\Aff(F)$ has the Riesz interpolation property for the usual order, as well as for the strict order, cf. Lemma 3.1 in \cite{EHS}. It follow therefore that there is a function $a_t \in \Aff(F)$ such that 
$$
c(y,t) + \delta < a_t(y) < d(y,t) - \delta
$$ 
for all $y \in F$. When $t = 1$ it follows in the same way that there is a function $a_1 \in \Aff(\Delta)$ such that
$$
c(x,1) + \delta \ < \ a_1(x) \ <  \ d(x,1) - \delta
$$
for all $x \in \Delta$. We can then construct a finite cover $V_i, i = 1,2,\cdots, N$, of $L$ by open sets and elements $t_i \in V_i$ with $t_1 =1$ such that
$$
\left| c(x,t) - c(x,t_i)\right| \leq \frac{\delta}{2}
$$
and
$$
\left| d(x,t) - d(x,t_i)\right| \leq \frac{\delta}{2}
$$
for all $t \in V_i$, all $x \in \Delta$ and all $i$. We arrange, as we can, that $1 \notin V_j , \ j \geq 2$. For $i \in \{2,3,\cdots , N\}$ we use Lemma \ref{21-12-20} to get $b_i\in \Aff(\Delta)$ such that $b_i|_F = a_{t_i}$, and we set $b_1 = a_1$. Let $\{\varphi_i\}_{i=1}^N$ be a partition of unity on $L$ which is subordinate to $\{V_i\}_{i=1}^N$. Define 
$a \in A(\Delta \times L)$ such that
$$
a(x,t) \ = \ \sum_{i=1}^N b_i(x)\varphi_i(t) \ .
$$ 
For $y \in F$,
\begin{align*}
&a(y,t) \ = \ \sum_{i=1}^N a_{t_i}(y)\varphi_i(t) \ \leq \ \sum_{i=1}^N (d(y,t_i ) - \delta)\varphi_i(t) \ \\
&\leq \ \sum_{i=1}^N (d(y,t) - \frac{\delta}{2}) \varphi_i(t) \ = \ d(y,t) - \frac{\delta}{2} \ ,
\end{align*}
and similarly $c(y,t) + \frac{\delta}{2} \ \leq \ a(y,t)$.
For $x \in \Delta$ we find that
\begin{align*}
&c(x,1) + \delta \ < \ a_1(x) \ = \ a(x,1) \ < \ d(x,1) - \delta \ .
\end{align*}
It follows that
$$
\Sigma(c^i)(x,t) < a(x,t) < \Sigma(d^j)(x,t)
$$ 
for all $i,j$ and all $(x,t) \in (F\times L) \cup (\Delta \times \{1\})$. Let $\epsilon > 0$. By Lemma \ref{16-12-20xx} there is a $g \in G$ such that
$$
\sup_{(x,t) \in \Delta \times L} \left|\Sigma(g)(x,t) \ - \ a(x,t)\right| \ \leq \ \epsilon \ .
$$
If $\epsilon$ is small enough it follows that $c^i \leq g \leq d^j$ in $G$ for all $i,j \in \{1,2\}$.
\end{proof}

Note that $\rho(G^+) = G^+$ where $\rho$ is the automorphism $\rho$ of $G$ defined in \eqref{rho}. It follows that $\rho$ is an automorphism of $(G,G^+)$ and then from Lemma \ref{18-12-20} and \cite{EHS} that there is an AF algebra $B$ whose $K_0$-group and dimension range is isomorphic to $(G ,G^+)$. Furthermore, it follows from \cite{E1} that $B$ is simple and stable, and that there is an automorphism $\gamma$ of $B$ such that $\gamma_* = \rho$ under the identification $K_0(B) = G$. It was shown by Nistor, \cite{Ni}, that $(G,G^+)$ has large denominators  and we can therefore use Lemma \ref{28-10-20} to choose $\gamma$ such that it has the following additional properties:
\begin{poem}\mbox{}\\[-\baselineskip]
  \begin{enumerate}\label{listrefx}
    \item The restriction map $\mu \ \mapsto \ \mu|_{B}$ is a bijection from the traces $\mu$ on $   B \rtimes_{\gamma} \mathbb Z$ onto the $\gamma$-invariant traces on $B$. \\
    \item $B \rtimes_{\gamma} \mathbb Z$ is stable. \\   
    \item $B \rtimes_{\gamma} \mathbb Z$ is $\mathcal Z$-stable.
    \end{enumerate}
\end{poem}

Set $C = B \rtimes_\gamma \mathbb Z$. No power of $\gamma$ is inner since $\gamma_*^k \neq \id_G$ for $k\neq 0$ and it follows therefore from \cite{Ki1} that $C$ is simple. The Pimsner-Voiculescu exact sequence shows that $K_1(C) = 0$ since $\id_G - \gamma_*$ is injective and that
$$
K_0(C) \ = \ \coker (\id_G-\gamma_*) \ = \ G/\left(\id_G - \gamma_*\right)(G) \ .
$$
Under this identification the map $K_0(B) \to K_0(C)$ induced by the inclusion $B \subseteq C$ is the quotient map $q : G \to  G/\left(\id_G - \gamma_*\right)(G)$. Hence $G^+/\left(\id_G - \gamma_*\right)(G) \subseteq K_0(C)^+$.

\begin{lemma}\label{27-10-20dxx} Let $d \in G$. Then $\sum_{m\in \mathbb Z} d_m = 0$ if and only if $d = (\id_G - \gamma_*)\left(g\right)$ for some $g \in G$.
\end{lemma}
\begin{proof}  It obvious that $\sum_{m\in \mathbb Z} d_m = 0$ when $d = (\id_G - \gamma_*)\left(g\right)$ for some $g \in G$. For the converse, choose $L\in \mathbb N$ so big that $x_n =0$ when $|n| \geq L$. Set $y_n = 0$ when $n > L$ and
$$
y_L = x_L , \ y_{L-1} = x_{L-1} + y_L, \ y_{L-2} = x_{L-2} +y_{L-1} , \ \text{etc} .
$$ 
Then 
$$
y_{-L} = x_{-L} + \sum_{i=1}^{2L} x_{-L+i} = \sum_{i=1}^{2L} x_{-L+i}  = 0 \ ,
$$
and hence $y_k = 0$ when $k < -L$. It follows that $(y_n)_{n \in \mathbb Z} \in G$ and $(\id_G -\gamma_*)\left((y_n)_{n\in \mathbb Z}\right) = (x_n)_{n \in \mathbb Z}$.
\end{proof}

It follows from Lemma \ref{27-10-20dxx} that we can define an injective homomorphism $S :  G/\left(\id_G - \gamma_*\right)(G) \to H$ such that 
$$
S\left(q(g)\right) \ = \ \sum_{m \in \mathbb Z} g_m \ \ \ \ \forall g \in G \ .
$$
$S$ is surjective since $S(q([[h]])) = h$ when $h \in H$, and hence an isomorphism with inverse $S^{-1}$ given by $S^{-1}(h) = q([[h]])$. Let $p \in B$ be a projection such that $[p] = [[u]]$ in $G$.

\begin{lemma}\label{18-12-20aaa} For $(y,t) \in F \times L$ and $x \in \Delta$ there are traces $\tau_{y,t}$ and $\tau_x$ on $B$ such that $\tau_{y,t}(p) = \tau_x(p) = 1$, and
$$
{\tau_{y,t}}_* (g) \ = \ \sum_{m \in \mathbb Z} \theta(g_m)(y)t^{m} 
$$
and
$$
{\tau_x}_*(g)  = \ \sum_{m \in \mathbb Z} \theta(g_m)(x)
$$
for all $g \in G$.
\end{lemma}
\begin{proof} This follows from Lemma \ref{23-10-20g}.
\end{proof}

\begin{lemma}\label{18-12-20b} Let $\phi : G \to \mathbb R$ be a positive homomorphism such that $\phi([[u]])= 1$.  Assume that $\phi \circ \gamma_* = s^{-1} \phi$ for some $s > 0$. Then $s \in L$ and when $s \neq 1$ it follows that $\phi = {\tau_{y,s}}_*$ for some $y \in F$, and when $s =1$ it follows that $\phi = {\tau_x}_*$ for some $x \in \Delta$.
\end{lemma}
\begin{proof} We claim that there is continuous linear map $\phi' : A(\Delta \times L) \to \mathbb R$ such that $\phi = \phi' \circ \Sigma$. Since $\theta(H)$ is dense in $\Aff(\Delta)$ there is a $c \in H$ such that $0 < \theta(c)(x) < N^{-1}$ for all $x \in \Delta$. Then $0 \leq N[[c]] \leq [[u]]$ in $G$ and hence $0 \leq \phi([[c]]) \leq \frac{1}{N}\phi\left([[u]]\right) = \frac{1}{N}$. Assume $g \in G$ and that $\Sigma(g) = 0$. Then $\pm g + [[c]]\in G^+$ and hence $-\frac{1}{N} \leq \phi(g) \leq \frac{1}{N}$. Letting $N \to \infty$ we conclude that $\phi(g) =0$, and it follows that there is a homomorphism $\phi' : \Sigma(G) \to \mathbb R$ such that $\phi' \circ \Sigma = \phi$. Let $h \in \Sigma(G)$; say $h = \Sigma(g)$, and let $k,l \in \mathbb N$ be natural numbers such that $|h(x,t)| < \frac{k}{l}$ for all $(x,t) \in \Delta \times L$. Then $k[[u]] \pm l g \in G^+$ and hence
$$
0  \ \leq  \ \phi(k[[u]] \pm l g) \ = \ k \pm l \phi'(h) \ .
$$
It follows that $\left|\phi'(h)\right| \leq \frac{k}{l}$, proving that $\phi'$ is Lipshitz continuous. Since $\Sigma(G)$ is dense in $A(\Sigma \times L)$ by Lemma \ref{16-12-20xx} it follows that $\phi'$ extends by continuity to a continuous linear map $\phi' : A(\Delta \times L) \to \mathbb R$, proving the claim. Let $T : A(\Delta \times L) \to A(\Delta \times L)$ denote the operator
$$
T(\psi)(x,t) = t^{-1}\psi(x,t) \ .
$$ 
Since $\Sigma \circ \gamma_*(g)(x,t) \ = \ t^{-1} \Sigma(g)(x,t)$ for all $g \in G$ and all $(x,t) \in \Delta \times L$, we find that
$$
s^{-1} \phi'( \Sigma(g)) \ = \ s^{-1}\phi(g) \ = \ \phi(\gamma_*(g)) \ = \ \phi' \circ \Sigma \circ \gamma_*(g)) \ = \ \phi'\left(T( \Sigma(g))\right) \ .
$$ 
It follows therefore from Lemma \ref{16-12-20xx} that 
\begin{equation}\label{18-20-12d}
s^{-1}\phi' = \phi' \circ T \ .
\end{equation}
When $a \in \Aff(\Delta)$ and $f \in C_{\mathbb R}(L)$ we denote by $a \otimes f$ the function $\Delta \times L \ni (x,t) \mapsto a(x)f(t)$. Assume $a(x) \geq 0$ for all $x \in \Delta$ and that $f(t) \geq 0$ for all $t \in L$. It follows then from Lemma \ref{16-12-20xx} and the definition of $G^+$ that $a \otimes f$ can be approximated by elements from $\Sigma(G^+)$, implying that $\phi'(a\otimes f) \geq 0$. There is therefore a bounded Borel measure $\mu_a$ on $L$ such that
$$
\phi'(a\otimes f) \ = \ \int_L f \ \mathrm{d}\mu_a \ \ \forall f \in C_{\mathbb R}(L) \ .
$$
It follows from \eqref{18-20-12d} that 
$$
 \int_L s^{-1} f(t) \ \mathrm{d}\mu_a(t) \ = \  \int_L t^{-1}f(t) \ \mathrm{d}\mu_a(t) 
$$
for all $f \in C_{\mathbb R}(L)$, and hence that 
$\mu_a(L \backslash \{s\}) = 0$. Since $\phi \neq 0$, not all $\mu_a$ can be zero and we conclude therefore that $s \in L$ and that
\begin{equation}\label{21-12-20d}
\phi'(a \otimes f) \ = \ \lambda(a)f(s) 
\end{equation}
for some $\lambda(a) \geq 0$ and all $f \in C_{\mathbb R}(L)$. Every element of $\Aff(\Delta)$ is the difference between two positive elements and it follows therefore that for all $a\in \Aff(\Delta)$ there is a real number $\lambda(a)$ such that \eqref{21-12-20d} holds for all $f \in C_{\mathbb R}(L)$. The resulting map $a \mapsto \lambda(a)$ is clearly linear and positive, and $\lambda(1) = 1$ since $1 = \phi([[u]]) = \phi'(1)$. This implies that there is an $x_0 \in \Delta$ such that $\lambda(a) = a(x_0)$ for all $a \in \Aff(\Delta)$. The elements of $\left\{a \otimes f : \ a \in \Aff(\Delta) , \ f \in C_{\mathbb R}(L) \right\}$ span a dense set in $A(\Delta \times L)$ and we find therefore that
\begin{equation}\label{18-12-20e}
\phi'(h) \ = \ h(x_0,s) \ \ \forall h \in A(\Delta \times L) \ .
\end{equation}
To see that $x_0 \in F$ when $s \neq 1$, assume for a contradiction that $x_0 \notin F$. By Lemma \ref{21-12-20} we can find $b \in \Aff(\Delta)$ such that $b|_F = 0$ and $b(x_0) > 0$, and since $s \neq 1$ we can find $f \in C_{\mathbb R}(L)$ such that $f(1) = 0$ and $f(s) = 1$. The function $b \otimes f$ vanishes on $(F \times L) \cup (\Delta \times \{1\})$. By Lemma \ref{16-12-20xx} and the definition of $G^+$ this implies that $\phi'(b \otimes f) = 0$, which contradicts \eqref{18-12-20e} since $b(x_0)f(s) > 0$. Hence $x_0 \in F$ when $s \neq 1$. It follows that
$$
\phi(g) \ = \ \phi'(\Sigma(g)) \ = \ \sum_{m \in \mathbb Z}\theta(g_m)(x_0)s^{m} \ ,
 $$
and that $x_0 \in F$ when $s \neq 1$.
\end{proof}

\subsection{The Elliott invariants of $pCp$ and $A$ are isomorphic}

\begin{lemma}\label{09-10-20axx} Set
$$
G^{++} = \{0\} \cup \left\{ g \in G : \ \sum_{m \in \mathbb Z} \theta(g_m)(x)\ > \ 0 \ \ \forall x \in \Delta \right\} \ .
$$
Then $K_0(C)^+ =  {G^{++}}/\left(\id_G - \gamma_*\right)(G)$ and 
$S$ takes $K_0(C)^+$ onto $H^+$.
\end{lemma}

\begin{proof} Let $g = (g_m)_{m\in \mathbb Z} \in G^{++} \backslash \{0\}$ and set $g^+ = [[\sum_{m \in \mathbb Z} g_m]] \in G$. Then $g^+ - g  \in (\id_G - \gamma_*)(G)$ by Lemma \ref{27-10-20dxx}, and since
$$
\sum_{m \in \mathbb Z} \theta(g^+_m)(x)t^{m} = \sum_{m \in \mathbb Z} \theta(g_m)(x) \geq \epsilon  \ ,
$$
for all $(x,t) \in \Delta \times L$, where $\epsilon = \min_{x \in \Delta} \sum_{m \in \mathbb Z} \theta(g_m)(x) > 0$, it follows that $g^+ \in G^+$. This shows that
$$
{G^{++}}/\left(\id_G - \gamma_*\right)(G) \subseteq {G^{+}}/\left(\id_G - \gamma_*\right)(G) \subseteq K_0(C)^+ \ .
$$
 Let $z \in K_0(C)^+\backslash \{0\} \subseteq G/\left(\id_G - \gamma_*\right)(G)$ and choose $g \in G$ such that $q(g) = z$. Let $x \in \Delta$. The trace $\tau_{x}$ from Lemma \ref{18-12-20b} is $\gamma$-invariant and there is therefore a trace $\tau_x'$ on $C$ such that $\tau_x'|_{B} = \tau_{x}$. Then ${\tau_x'}_*(z) > 0$ since $z > 0 $ in $K_0(C)$ and $C$ is simple. Hence
$$
0 <{\tau_x'}_*(z) = {\tau_{x}}_*(g) = \ \sum_{m \in \mathbb Z} \theta(g_m)(x) \ .
$$
It follows that $g\in G^{++}$ and $z \in {G^{++}}/\left(\id_G - \gamma_*\right)(G)$.  We conclude that 
$K_0(C)^+ =  {G^{++}}/\left(\id_G - \gamma_*\right)(G)$ and $S(K_0(C)^+) \subseteq H^+$. When $h \in H^+ \backslash \{0\}$, set $z = [[h]]$. Then $z \in G^{++}$, $q(z) \in K_0(C)^+$ and $S(q(z)) = h$. 
\end{proof}

\begin{lemma}\label{02-12-20bx} Let $p \in B$ be a projection representing $[[u]] \in G= K_0(B)$. The Elliott invariants of $A$ and $p\left( B \rtimes_\gamma \mathbb Z\right)p$ are isomorphic.
\end{lemma}
\begin{proof} Since $K_1(A) = K_1(pCp) = 0$ it suffices to supplement the isomorphism $S : (K_0(C),K_0(C)^+,[[u]]) \to (H,H^+,u)$ of partially ordered groups with order unit with an affine homeomorphism $T : T(pCp) \to T(A)$ such that
\begin{equation}\label{20-12-20}
T(\tau)_*(S(z)) = \tau_*(z) \ \ \ \forall (\tau,z) \in T(pCp) \times K_0(C) \ .
\end{equation}
Let $\tau \in T(pCp)$. By Proposition 4.7 in \cite{CP} there is a unique trace $\tau'$ on $C$ such that $\tau'|_{pCp} = \tau$. Then $\tau'|_B$ is a $\gamma$-invariant trace on $B$ such that $\tau'(p) = 1$. It follows from Lemma \ref{23-04-21} that there is an $x \in \Delta$ such that ${(\tau'|_B)}_* = {\tau_x}_*$ on $K_0(B)$. 
 Since $x$ is uniquely determined by $\tau$ we can define $T : T(pCp) \to T(A)$ such that $T(\tau) = \tr_x$. Let $e \in K_0(A), \ 0 \leq e \leq [1]$. Then $0 \leq [[e]] \leq [[u]]$ in $G$ and hence $[[e]]$ is represented in $K_0(B)$ by a projection $e'$ from $pBp \subseteq pCp$. We find that
 $$
 {\tr_x}_*(e) = \theta(e)(x) = {\tau_x}_* ([[e]]) = {(\tau'|_B)}_*([[e]]) = \tau(e') \ ;
 $$ 
 an equality which shows that $T$ is affine and continuous. To see that $T$ is surjective let $x \in \Delta$. The trace $\tau_x$ from Lemma \ref{18-12-20aaa} is $\gamma$-invariant and defines therefore a trace $\tau'$ on $C$ such that $\tau'|_B = \tau_x$. Then $\tau'(p) = {\tau_x}_*([[u]]) = 1$ and hence $\tau'|_{pCp} \in  T(pCp)$ and $T\left(\tau'|_{pCp}\right) = \tr_x$. To see that $T$ is also injective consider to traces $\tau_i, i = 1,2$, in $T(pCp)$ and assume that $T(\tau_1) = T(\tau_2)$. Then ${(\tau'_1|_B)}_* = {(\tau'_2|_B)}_*$ on $G$ and hence $\tau'_1|_B = \tau'_2|_B$ by Lemma \ref{23-10-20g}. It follows from the first of the additional properties \ref{listrefx} that $\tau'_1 = \tau'_2$ and hence also that $\tau_1 = {\tau'_1}|_{pCp} = {\tau'_2}|_{pCp} = \tau_2$. We conclude that $T$ is an affine homeomorphism. Consider $\tau \in T(pCp)$ and $e \in K_0(A), \ 0 \leq e \leq 1$. Then $T(\tau) = \tr_x$ for some $x \in \Delta$ and, as observed above, $[[e]] \in G$ is represented by a projection $e'\in pBp$ such that ${\tr_x}_*(e) = \tau(e')$. Hence  
\begin{align*}
& T(\tau)_*\left( S(S^{-1}(e))\right)  = {\tr_x}_*(e) = \tau(e') = \tau_*(q([[e]]) =\tau_* \left(S^{-1}(e)\right) \ .
\end{align*}
Since the elements of $\left\{ S^{-1}(e) : \ e \in K_0(A), \ 0 \leq e \leq [1]\right\}$ generate $K_0(pCp)$ as a group, we conclude that \eqref{20-12-20} holds.
\end{proof}

\subsection{Conclusion}

\begin{lemma}\label{20-12-20a} $A$ is $*$-isomorphic to $p(B \rtimes_\gamma \mathbb Z)p$.
\end{lemma}
\begin{proof} By Lemma \ref{02-12-20bx} it remains only to show that $A$ and $pCp$ belong to a class of simple $C^*$-algebras for which the Elliott invariant is complete. For this we appeal to Corollary D of \cite{CETWW} and we must therefore check that both algebras are separable, unital, nuclear, $\mathcal Z$-stable and satisfy the UCT. Only $\mathcal Z$-stability is not well-known. For $A$ this follows from Theorem A and Corollary C in \cite{CETWW} and for $pCp$ it follows from the third of the Additional properties \ref{listrefx} combined with Corollary 3.2 in \cite{TW}. 
\end{proof}

\begin{remark}\label{14-12-20gx} Assume that $A$ is UHF. Then the algebras $pCp$ and $A$ occurring in the last proof only have one trace state and there is an alternative way to deduce the isomorphism $A \simeq pCp$. The path through the litterature takes more space to explain, but presents presumably a shorter argument: By Corollary 6.2 in \cite{MS} it suffices to show that both algebras are unital, separable, simple, infinite dimensional, nuclear, quasi-diagonal, satisfy the UCT and have strict comparison. Many of these properties are well-known for both algebras. What remains is to explain why they are quasi-diagonal and have strict comparison. Since $pCp$ and $A$ are both exact, simple and unital it follows from Corollary 4.6 in \cite{Ro} that $\mathcal Z$-stability implies strict comparison and hence it suffices to argue that both algebras are quasi-diagonal and $\mathcal Z$-absorbing. Concerning quasi-diagonality it is well-known that AF algebras have this property and that it is inherited by subalgebras, so a straightforward application of \cite{Br} shows that both algebras are quasi-diagonal. $C$ is $\mathcal Z$-absorbing by the third of the additional properties \ref{listrefx} and hence so is $pCp$ by Corollary 3.2 in \cite{TW} since $pCp$ is stably isomorphic $C$ by \cite{B}. Finally, since $A$ is approximately divisible it is also $\mathcal Z$-absorbing by Theorem 2.3 in \cite{TW}.
\end{remark}

\begin{remark}\label{28-04-21a} We haven't used the second of the Additional properties \ref{listrefx}. It was added because stability of $B \rtimes_{\gamma} \mathbb Z$ combined with Lemma \ref{20-12-20a} and \cite{B} implies that $B \rtimes_{\gamma} \mathbb Z \simeq A \otimes \mathbb K$; a fact which is nice to know, but it is not needed for the proof of Theorem \ref{20-12-20c}. \footnote{It has been pointed out to me that stability of $B \rtimes_{\gamma} \mathbb Z$ also follows from \cite{HR}.} 
\end{remark}

The dual action on $B \rtimes_\gamma \mathbb Z$ defines by restriction a $2\pi$-periodic flow on $p(B \rtimes_\gamma \mathbb Z)p$ which we denote by $\widehat{\gamma}$.

\begin{lemma}\label{02-12-20d} For $\beta \in \mathbb R$ there is a $\beta$-KMS state for $\widehat{\gamma}$ on $p(B \rtimes_\gamma \mathbb Z)p$ if and only if $\beta \in K$. For $\beta \in K \backslash \{0\}$ the simplex of $\beta$-KMS states for $\widehat{\gamma}$ is affinely homeomorphic to $F$ and for $\beta =0$ it consists of all trace states on $p(B \rtimes_\gamma \mathbb Z)p$ and hence is affinely homeomorphic to $\Delta$. 
\end{lemma} 
\begin{proof} Let $\beta \in K \backslash \{0\}$ and $y \in F$. It follows from Lemma \ref{26-10-20} that we can define a $\beta$-KMS state on $p(B \times_\gamma \mathbb Z)p$ by $\tau_{y,e^{\beta}}\circ P|_{p(B \rtimes_{\gamma} \mathbb Z)p}$.  The resulting map, from $F$ to the simplex of $\beta$-KMS states for $\widehat{\gamma}$ on $p(B \times_\gamma \mathbb Z)p$,
is clearly continuous, affine and injective. To see that it is surjective, let $\omega$ be a $\beta$-KMS state for $\widehat{\gamma}$. From Theorem 2.4 in \cite{Th2}, or more precisely from the part of that theorem which follows from Remark 3.3 in \cite{LN}, combined with Lemma \ref{26-10-20}, it follows that there is trace $\tau$ on $B$ such that $\tau \circ \gamma = e^{-\beta} \tau$, $\tau(p) = 1$ and $\omega = \tau \circ P|_{p(B\rtimes_\gamma \mathbb Z)p}$. It follows from Lemma \ref{23-04-21} that $\tau_* = {\tau_{y,e^{\beta}}}_*$ on $G$ for some $y \in F$ and then from Lemma \ref{23-10-20g} that $\tau = \tau_{y,e^{\beta}}$. We conclude that $\omega$ is the image of $y \in F$ under the map we consider. This shows that the map is an affine homeomorphism. The same argument, using the traces $\tau_x$ from Lemma \ref{18-12-20aaa}, shows that the simplex of $0$-KMS states for $\widehat{\gamma}$ is affinely homeomorphic to $\Delta$.

It remains to show that there are no $\beta$-KMS states for $\widehat{\gamma}$ unless $\beta \in K$. The argument for this is almost identical to one we have given above: Assume that there is $\beta$-KMS state for $\widehat{\gamma}$. It follow from Remark 3.3 in \cite{LN} that this $\beta$-KMS state is the restriction to $p(B \rtimes_\gamma \mathbb Z)p$ of a $\beta$-KMS weight for the dual action on $B \rtimes_\gamma \mathbb Z$ and then from Lemma \ref{26-10-20} that there is a trace $\tau$ on $B$ such that $\tau \circ \gamma = e^{-\beta} \tau$ and $\tau(p)=1$. Then $\tau_* : G \to \mathbb R$ is a positive homomorphism such that $\tau_*\circ \gamma_* = e^{-\beta}\tau_*$ and $\tau_*([[u]]) =1$. It follows from Lemma \ref{23-04-21} that $\beta \in K$.
\end{proof}

Now Theorem \ref{20-12-20c}, with $K$ compact, follows from Lemma \ref{02-12-20d} and Lemma \ref{20-12-20a}.

\section{Proof of Theorem \ref{20-12-20c} when $K$ is unbounded}

Let $K$ be the closed lower bounded set of real numbers from Theorem \ref{20-12-20c} and set $L = \left\{e^{\beta} : \ \beta \in K\right\}$; a closed subset of $]0,\infty[$ bounded away from $0$. We have already covered the case where $K$ is also bounded above, so here we assume that $K$ and also $L$ is unbounded.

\subsection{The second construction}
Let $P_{\Delta}(t)$ denote the vector space of Laurent polynomials in $t$ with coefficients from $\Aff(\Delta)$. Thus $P_{\Delta}(t)$ consists of the functions $f : \Delta \times ]0,\infty[ \to \mathbb R$ of the form
\begin{equation}\label{22-04-21}
f(x,t) = \sum_{n=-\infty}^\infty a_n(x)t^n \ ,
\end{equation}
where $a_n \in \Aff(\Delta)$ for all $n$ and only finitely many $a_n$'s are non-zero. Given $f \in P_\Delta(t)$ as in \eqref{22-04-21} we set
$$
\mathbb L(f) = \max \left\{ n \in \mathbb Z: \ a_n \neq 0 \right\} \ 
$$ 
when $f \neq 0$, and let $\mathbb L(0) = -\infty$. We will refer to $\mathbb L(f)$ as the degree of $f$ and the element $a_{\mathbb L(f)} \in \Aff(\Delta)$ as the leading coefficient of $f$.

Let $F$ be a closed non-empty face in $\Delta$. For $f \in P_\Delta(t)$ we write $0 \prec f$ when there is an $\epsilon >0$ such that
$$
t^{-\mathbb L(f)}f(x,t) \geq \epsilon \  \ \ \ \forall (x,t) \in (F \times L) \cup (\Delta \times \{1\}) \ .
$$
Thus $0 \prec f$ if and only if $f$ is strictly positive on $(F \times L) \cup (\Delta \times \{1\})$ and the leading coefficient of $f$ is strictly positive on $\Delta$. Given $f,g \in P_\Delta(t)$, we write $f \prec g$ when $0\prec g-f$. As in Section \ref{first} we set $G = \oplus_{\mathbb Z} H$ and define a homomorphism $\Sigma : G \to P_\Delta(t)$ such that
$$
\Sigma(g)(x,t) = \sum_{m \in \mathbb Z} \theta(g_m)(x)t^m \ .
$$
Set
$$
G^+ = \left\{ g \in G: \ 0 \prec \Sigma(g) \right\} \cup \{0\} \ .
$$
We aim now to show that $(G,G^+)$ is a Riesz group, and hence also a dimension group in the sense of \cite{EHS}. This means that we must prove that 
\begin{itemize}
\item $G^+ + G^+ \subseteq G^+$,
\item $G^+ \cap (-G^+) = \{0\}$,
\item $ g\in G, \ n \in \mathbb N \backslash \{0\}, \ ng \in G^+ \ \Rightarrow \ g \in G^+$, 
\item $G^+ - G^+ = G$, and
\item $G$ has the Riesz interpolation property.
\end{itemize}
The first three items are easily checked. To show that $G^+ - G^+ = G$, let $g \in G$. Choose $l \in \mathbb N$ bigger than any $n \in \mathbb N$ for which $\theta(g_n)$ is non-zero and set $l =0$ if $\theta(g_n) = 0$ for all $n$. Choose $N \in \mathbb N$ such that $Nt^{l}  > \sum_{n \in \mathbb Z} \left|\theta(g_n)(x)\right|t^n$ for all $n \in \mathbb N$ and $(x,t) \in \Delta \times L$. Define $v \in G$ such that
$$
v_n = \begin{cases} 0, \ n \neq l \ ,\\ Nu , \ n = l \ . \end{cases} 
$$
Then $v, v-g \in G^+$ and $ g = v-(v-g)$. 

It remains to show that $G$ has the Riesz interpolation property. The next subsection is devoted to that.

\subsubsection{$(G,G^+)$ is a Riesz group}

For $g \in G$, set
$$
\mathbb L(g) = \max \{n \in \mathbb Z: g_n \neq 0 \}
$$
when $g \neq 0$, and $\mathbb L(0) = -\infty$. Then 
$\mathbb L(\Sigma(g)) \leq \mathbb L(g)$.

\begin{lemma}\label{22-04-21x} Let $a  \in \Aff(\Delta)$ and let $f: L \to\mathbb R$ be a continuous function such that $\lim_{t \to \infty} f(t) = 0$. Let $\epsilon > 0$ and $N  \in \mathbb N$. There is a $g \in G$ such that
$\mathbb L(g) \leq -N$ and
$$
\left| \Sigma(g)(x,t) - f(t)a(x)\right| \leq  \epsilon \ \ \forall (x,t) \in \Delta \times L \ .
$$
\end{lemma}
\begin{proof} By the Stone-Weierstrass theorem there is a real Laurent polynomial $Q(t) = \sum_{n= -M}^{-N} c_nt^{n}$ of degree at most $-N$ such that 
$$
\left|a(x)\right|\left|Q(t) -f(t)\right| \leq \frac{\epsilon}{2}
$$ 
for all $x \in \Delta$ and all $t \in L$. Since $\theta(H)$ is dense in $\Aff(\Delta)$ there are elements $h_n \in H$ such that
$$
\left|\theta(h_n)(x) - c_na(x) \right|t^{n} \leq \frac{\epsilon}{2(M-N)}
$$
for all $(x,t) \in \Delta \times L$ and all $n \in \{-M,-M+1,\cdots , -N\}$. Define $g \in G$ such that
$$
g_n = \begin{cases} 0 \ , \ & n \leq -M-1 \ \text{or} \ n \geq -N+1 \ ,\\ h_{n} \ , \ & \ -M \leq n \leq -N\ . \end{cases}
$$
Then $\mathbb L(g) \leq -N$ and
\begin{align*}
& \left| \Sigma(g)(x,t) - f(t)a(x)\right| \\
& \leq \left| \sum_{n=-M}^{-N} \theta(h_n)(x)t^{n} - \sum_{n=-M}^{-N} c_na(x)t^{n}\right| + \left| Q(t)a(x) - f(t)a(x)\right| \leq \epsilon \ .   
\end{align*}
\end{proof}

\begin{lemma}\label{05-05-21a} Let $g  \in G$ and let $F \subseteq \mathbb Z$ be a finite set. Assume that $g_n = 0, \ |n| \geq J\geq 1$. Let $N \in \mathbb N$ such that $J-N -d < -J$ for all $d \in F \cup \{0\}$. Let $q$ be a real-valued Laurent polynomial of degree at most $-N$. For any $\epsilon > 0$ there is an element $g' \in G$ such that $g'_n = g_n$ when $n\geq  -J$, and 
$$
\left|t^{-d}(1+q(t))\Sigma(g)(x,t) - t^{-d}\Sigma(g')(x,t) \right| \leq \epsilon
$$
for all $d \in F$ and all $(x,t) \in \Delta \times L$. 
\end{lemma}
\begin{proof} By assumption we can write
$$
q(t)\Sigma(g)(x,t) = \sum_{n= B}^A a_n(x) t^n \ ,
$$
where $a_n \in \Aff(\Delta)$ for all $n$ and $B < A \leq J-N$. Since $\theta(H)$ is dense in $\Aff(\Delta)$ and since $J-N -d < -J < 0$ we can find $h_n \in H, \ B \leq n \leq A$, such that
$$
\left| t^{-d} q(t)  \sum_{n= B}^A a_n(x) t^n -  t^{-d}\sum_{n= B}^A \theta(h_n)(x) t^n\right| \leq \epsilon
$$
for all $d \in F$ and all $(x,t) \in \Delta \times L$. Define $g' \in G$ such that $g'_n = h_n, \ B \leq n \leq A$, and $g'_n = g_n$ for $n \notin [B,A]$. Then
\begin{align*}
&\left|t^{-d}(1+q(t))\Sigma(g)(x,t) - t^{-d} \Sigma(g')(x,t)\right| \\
&= \left| t^{-d} q(t)  \sum_{n= B}^A a_n(x) t^n -  t^{-d}\sum_{n= B}^A \theta(h_n)(x) t^n\right| \leq \epsilon \ 
\end{align*}
for all $d \in F$ and all $(x,t) \in \Delta \times L$. 
\end{proof}

\begin{lemma}\label{06-05-21a} Let $F : \Delta \times L \to \mathbb R$ be a continuous function such that $\Delta \ni x \mapsto F(x,t)$ is affine for all $t \in L$ and $\lim_{t \to \infty} F(x,t) = 0$ uniformly in $x$. Let $J \in \mathbb N$. For every $\delta > 0$ there is a $g \in G$ such that $\mathbb L(g) \leq  -J$, and
$$
\left|\Sigma(g)(x,t) - F(x,t)\right| \leq \delta
$$
for all $(x,t) \in \Delta \times L$.
\end{lemma}
\begin{proof} Let $\kappa > 0$. There is then an $R > 0$ such that $\left|F(x,t)\right| \leq  \kappa$ for all $x \in \Delta$ and all $t \geq R$. For each $t' \in L \cap \left]0, R\right]$, there is an element $h_{t'} \in H$ such that
$$
\left|F(x,t) - \theta(h_{t'})(x)\right| < \kappa \ 
$$
for all $x \in \Delta$ and all $t$ in an open neighborhood of $t'$. By compactness of $L \cap \left]0, R\right]$, we can therefore find open sets $U_i$ in $\mathbb R$ and elements $h_i \in H, i = 1,2, \cdots , n$, such that
$$
\left|F(x,t) - \theta(h_i)(x)\right| \leq \kappa
$$
for all $x \in \Delta \times (L\cap U_i), \ i =1,2,\cdots , n$, and $L \cap \left]0, R\right] \subseteq \bigcup_{i=1}^n U_i$. Let $\varphi_i : \mathbb R \to \mathbb R , \ i = 1,2,\cdots, n$, be continuous non-negative compactly supported functions such that $\supp \varphi_i  \subseteq U_i$ for all $i$, $\sum_{i=1}^n \varphi_i \leq 1$ and $\sum_{i=1}^n \varphi_i(x) = 1$ for all $x \in  L \cap \left]0, R\right]$. By Lemma \ref{22-04-21x} there are elements $g_i \in G$ such that $\mathbb L(g_i) \leq -J$ and
$$
\left| \Sigma(g_i)(x,t) - \varphi_i(t)\theta(h_i)(x)\right| \leq \frac{\kappa}{n} 
$$
for all $(x,t) \in \Delta \times L$. Set $g = \sum_{i=1}^n g_i$ and note that $\mathbb L(g) \leq -J$. Furthermore,
\begin{align*}
&\left| \Sigma(g)(x,t) - F(x,t)\right| \\
& \leq \left| (1- \sum_{i=1}^n\varphi_i(t)) F(x,t) \right| + \sum_{i=1}^n \left|\Sigma(g_i)(x,t) - \varphi_i(t)F(x,t)\right| \\
& \leq \kappa + \sum_{i=1}^n \left(\left|\Sigma(g_i)(x,t) - \varphi_i(t)\theta(h_i)(x)\right| + \varphi_i(t) \left|\theta(h_i)(x) - F(x,t)\right|\right) \\
& \leq 3 \kappa
\end{align*} 
for all $(x,t) \in \Delta \times L$.
\end{proof}

\begin{lemma}\label{21-04-21} $(G,G^+)$ is a dimension group.
\end{lemma}
\begin{proof}
 As observed above it remains only to establish the Riesz interpolation property. Let $u^i,v^j \in G, \ i,j \in \{1,2\}$, satisfy that $u^i \leq v^j$ in $G$ for all $i,j$. We must find $g' \in G$ such that $u^i \leq g' \leq v^j$ for all $i,j$. If $u^{i'} = v^{j'}$ for some $i',j' \in \{1,2\}$, set $g' = v^{j'}$. Then $u^i \leq g' \leq v^j$ for all $i,j$ and we are done. Assume therefore that $\Sigma(u^i) \prec \Sigma(v^j)$ for all $i,j \in \{1,2\}$. 

\begin{obs}\label{05-05-21} There is a $g \in G$, an $R > 1$ and an $\epsilon > 0$, such that 
$$
t^{-l_i}(\Sigma(g)(x,t) - \Sigma(u^i)(x,t)) > \epsilon 
$$
for all $(x,t) \in F \times (L \cap [R,\infty))$, $i=1,2$, where 
$$
l_i = \mathbb L( \Sigma(g) - \Sigma(u^i)) \ ,
$$ 
and such that
$$
t^{-l'_j}(\Sigma(v^j)(x,t) - \Sigma(g)(x,t)) > \epsilon 
$$
for all $(x,t) \in F \times (L \cap [R,\infty))$, $j = 1,2$, where 
$$
l'_j = \mathbb L( \Sigma(v^j) - \Sigma(g)) \ .
$$ 
\end{obs}

\begin{proof} Choose $N \in \mathbb N$ such that $u^i_k = v^j_k = 0$ for all $i,j \in \{1,2\}$ when $|k|\geq N$. Define $\overline{u}^i \in \oplus_{k=0}^\infty \theta(H)$ such that
$$
\overline{u}^i_k = \theta(u^i_{N-k}) \ ,
$$
and $\overline{v}^j \in \oplus_{k=0}^\infty \theta(H)$ such that
$$
\overline{v}^j_k = \theta(v^j_{N-k}) \ .
$$
We consider here $\theta(H)$ is a partially ordered group in the order inherited from the strict order of $\Aff(F)$; i.e. $\theta(H)^+$ consists of $0$ and the elements $f \in \theta(H)$ such that $f(x) > 0$ for all $x \in F$. It is well-known that $\theta(H)$ then has the Riesz interpolation property, cf. Lemma 3.1 and Lemma 3.2 in \cite{EHS}. The direct sum $\oplus_{k=0}^\infty \theta(H)$, in turn, is given the corresponding lexicographic order $<_{lex}$, where $(h_k)_{k=0}^{\infty} <_{lex} (h'_k)_{k=0}^\infty$ means that $h_k = h'_k$ for all $k$, or $h_{k_0} < h'_{k_0}$ in $\theta(H)$ where $k_0 = \min \{ k : \ h_k \neq h'_k \}$. Since $ \Sigma(u^i) \prec \Sigma(v^j)$ for all $i,j \in \{1,2\}$ we have that
$$
\overline{u}^i <_{lex} \overline{v}^j
$$
for all $i,j \in \{1,2\}$. By Theorem 3.10 in \cite{E2} the group $\oplus_{k=0}^\infty \theta(H)$ has the Riesz interpolation property in the  lexicographic order\footnote{Theorem 3.10 in \cite{E2} is very general and the argument concerning Riesz groups uses the decomposition property. However, a major step in the proof of Lemma 3.2 in \cite{BEK} gives the arguments directly for the interpolation property, albeit in the colexicographic order.} and there is therefore an element $h'' \in \oplus_{k=0}^\infty \theta(H)$ such that $\overline{u}^i <_{lex} h'' <_{lex} \overline{v}^j$ for all $i,j \in \{1,2\}$. If $h'' \notin \{\overline{u}^1,\overline{u}^2,\overline{v}^1,\overline{v}^2\}$ we define $g \in G$ such that $g_k = 0$ when $k > N$ and such that $\theta(g_k) = h''_{N-k}$ when $k \leq N$, and $g$ will then have the desired properties because the leading coefficients of $\Sigma(g)-\Sigma(u^i)$ and $\Sigma(v^j) -\Sigma(g), \ i,j =1,2$, will all be strictly positive. If $h'' = \overline{u}^{1}$ it follows that $\overline{u}^{2} <_{lex}\overline{u}^1$ which implies that $\Sigma(u^2) = \Sigma(u^1)$ or that there is an $R' > 1$ and a $\delta > 0$ such that
$$
t^{-l}(\Sigma(u^1)(x,t) - \Sigma(u^2)(x,t))  \geq \delta > 0
$$
for all $(x,t) \in F \times [R',\infty)$ where $l = \mathbb L\left(\Sigma(u^1) - \Sigma(u^2)\right)$. In both these cases we choose, as we can, an element $d \in H$ such that
$$
0 < \theta(d)(x)t^{-N-1} < \frac{1}{2}(\Sigma(v^j)(x,t) - \Sigma(u^1)(x,t))
$$
for all $(x,t) \in (F \times L)\cup (\Delta \times \{1\})$, $j = 1,2$, and define $g \in G$ such that 
$$
g_k = \begin{cases} u^1_k , & \ -N \leq k \leq N, \\ d , & \ k = -N -1, \\ 0, & \ k\notin [-N-1,N]  \  . \end{cases} \
$$ 
Then $g$ will have the desired property.
The cases where $h'' \in\{\overline{u}^2,\overline{v}^1,\overline{v}^2\}$ are handled in a similar way.
\end{proof}

We fix now $g$, $R$ and $\epsilon$ as in Observation \ref{05-05-21} and choose $J \in \mathbb N$ such that $g'_n = 0$ when $|n| > J$ and $g' \in \left\{g,v^1,v^2,u^1,u^2\right\}$. And we choose $J' \in \mathbb N$ such that
$$
J' > 2J-d
$$
for all $d \in \{0, l_1,l_2,l'_1,l'_2\}$.

\begin{obs}\label{06-05-21} There is an element $g^{(1)} \in G$ and an $\epsilon_1 > 0$ such that $\mathbb L(g^{(1)}) \leq -J'$ and
$$
t^{-d}\Sigma(u^i)(x,t) + \epsilon_1 < t^{-d}\Sigma(g^{(1)})(x,t) < t^{-d} \Sigma(v^j)(x,t) -\epsilon_1
$$
for all $(x,t) \in (F \times L \cap ]0,R+1]) \cup (\Delta \times \{1\})$ and all $d \in \{l_1,l_2,l'_1,l'_2\}$.
\end{obs}
\begin{proof} Since $\Sigma(u^i)(x,t)  < \Sigma(v^j)(x,t)$ for all $(x,t)$ in the compact set $(F \times L \cap ]0,R+1]) \cup (\Delta \times \{1\})$ there is a $\delta > 0$ such that
$$
t^{-d}\Sigma(u^i)(x,t) + \delta <  t^{-d} \Sigma(v^j)(x,t) - \delta 
$$
for all $(x,t) \in (F \times L \cap ]0,R+1]) \cup (\Delta \times \{1\})$ and all $d \in \{l_1,l_2,l'_1,l'_2\}$. By using the Riesz interpolation property of $\Aff(F)$ and $\Aff(\Delta)$ as is in the proof of Lemma \ref{18-12-20} we construct a continuous function 
$$
a: \Delta \times (L \cap  ]0,R+1]) \to \mathbb R
$$ 
such that $x \mapsto a(x,t)$ is affine for all $t \in  L \cap  ]0,R+1]$ and 
$$
t^{-d}\Sigma(u^i)(x,t) + \delta < t^{-d} a(x,t) < t^{-d} \Sigma(v^j)(x,t) - \delta 
$$
for all $(x,t) \in (F \times L \cap ]0,R+1]) \cup (\Delta \times \{1\})$ and all $d \in \{l_1,l_2,l'_1,l'_2\}$. Let $r$ be the largest element in $L \cap  ]0,R+1]$ and extend $a$ to $\Delta \times L$ such that
$$
a(x,t) = \begin{cases} (r+1 - t)a(x,r), \ & t \in [r,r+1] \ ,\\ 0, \ & t \geq r+1 \ . \end{cases}
$$
It follows from Lemma \ref{06-05-21a} that there is an element $g^{(1)} \in G$ such that $\mathbb L(g^{(1)}) \leq -J'$ and
$$
\left|t^{-d} \Sigma(g^{(1)})(x,t) - t^{-d}a(x,t)\right| \leq \frac{\delta}{2}
$$
for all $(x,t)\in (F \times L \cap ]0,R+1]) \cup (\Delta \times \{1\})$ and all $d \in \{l_1,l_2,l'_1,l'_2\}$. Set $\epsilon_1 = \frac{\delta}{2}$.
\end{proof}

Let $\psi :]0,\infty[ \to [0,1]$ be the continuous function such that $\psi(t) = 1, t\leq R$, $\psi(t) = 0, \ t \geq R+1$, and $\psi$ is linear on $[R,R+1]$. Choose $\kappa > 0$ such that
$$
4 \kappa < \min \{\epsilon,\epsilon_1\} \ 
$$
where $\epsilon >0$ comes from Observation \ref{05-05-21} and $\epsilon_1 >0$ from Observation \ref{06-05-21}.
Find a real-valued Laurent polynomial $q$ of degree no more than $-J'$ such that
\begin{equation}\label{06-05-21g}
\left|t^{-d}\psi(t)\Sigma(g^{(1)})(x,t) -t^{-d}q(t)\Sigma(g^{(1)})(x,t)\right| \leq \kappa
\end{equation}
for all $(x,t) \in \Delta \times L$ and all $d\in \{l_1,l_2,l'_1,l'_2\}$. This is possible because $\psi|_L \in C_0(L)$, so that $\psi|_L$ can be approximated by Laurent polynomials of degree as most $-J'$, and
$$
\sup_{(x,t)\in \Delta \times L} \left|t^{-d}\Sigma(g^{(1)})(x,t)\right| < \infty \ 
$$
since $\mathbb L(g^{(1)}) - d \leq 0$. Since 
$$
\sup_{(x,t) \in \Delta \times L}  t^{-l_i}\left(\Sigma(g)(x,t) - \Sigma(u^i)(x,t)\right) < \infty 
$$
and
$$
\sup_{(x,t) \in \Delta \times L}  t^{-l'_j}\left(\Sigma(v^j)(x,t) -\Sigma(g)(x,t)\right) < \infty 
$$
for all $i,j$, we can also arrange that 
\begin{equation}\label{07-05-21b}
\left| (q(t) - \psi(t))t^{-l_i}\left(\Sigma(g)(x,t) - \Sigma(u^i)(x,t)\right)\right|  < \kappa
\end{equation}
and
\begin{equation}\label{07-05-21c}
\left| (q(t) - \psi(t))t^{-l'_j}\left(\Sigma(v^j)(x,t) -\Sigma(g)(x,t)\right)\right|  < \kappa
\end{equation}
for all $(x,t) \in \Delta \times L$ and all $i,j$.

We apply now Lemma \ref{05-05-21a} to get an element $g^{(2)} \in G$ such that $g^{(2)}_n = g_n$ for all $-J \leq n \leq J$, $g^{(2)}_n = 0, \ n\geq J+1$, and
\begin{equation}\label{06-05-21d}
\left|t^{-d}(1-q(t))\Sigma(g)(x,t) - t^{-d}\Sigma(g^{(2)})(x,t)\right| \leq \kappa
\end{equation}
for all $(x,t) \in \Delta \times L$, and all $d \in \{0,l_1,l_2,l'_1,l'_2\}$. 

\begin{obs}\label{07-05-21} There is an element $g^{(3)} \in G$ such that $\mathbb L(g^{(3)}) < -J$ and
\begin{equation}\label{06-05-21e}
\left|t^{-d}q(t)\Sigma(g^{(1)})(x,t) - t^{-d}\Sigma(g^{(3)})(x,t)\right| \leq \kappa
\end{equation}
for all $(x,t) \in \Delta \times L$ and all $d \in \{l_1,l_2,l'_1,l'_2\}$.
\end{obs}
\begin{proof} Write 
$$
q(t)\Sigma(g^{(1)})(x,t) = \sum_{-M}^{-2J'}a_n(x)t^n
$$ 
where $M > 2J'$. Let $\kappa' > 0$. Since $\theta(H)$ is dense in $\Aff(\Delta)$ we can find $g^{(3)}\in G$ such that $g^{(3)}_n = 0$ when $n \notin [-M,-2J']$ and
$$
\left|\theta(g^{(3)}_n)(x) - a_n(x)\right| \leq \kappa'
$$
for all $x \in \Delta$ and all $-M \leq n \leq -2J'$. Then
\begin{align*}
&\left|t^{-d}q(t)\Sigma(g^{(1)})(x,t) - t^{-d}\Sigma(g^{(3)})(x,t)\right|\\
& \leq \kappa' (M-2J')\sup\left\{ t^n : \ -M-d\leq n \leq - 2J' -d,\ t \in L \right\}
\end{align*}
for all $(x,t) \in \Delta \times L$  and all $d \in \{l_1,l_2,l'_1,l'_2\}$. Since $-2J'-d < 0$,
$$
\sup\left\{ t^n : \ -M-d\leq n \leq - 2J' -d,\ t \in L \right\}
$$
is finite for all $d$ and we can therefore arrange that \eqref{06-05-21e} holds.
\end{proof}

 Set $g^{(4)} = g^{(3)}+ g^{(2)}$. Since $g^{(4)}_n = g_n$ for all $-J \leq n \leq J$ and $g^{(4)}_n= 0, \ n \geq J+1$, we find that
$$
\mathbb L(\Sigma(g^{(4)}) -\Sigma(u^i)) = \mathbb L(\Sigma(g) -\Sigma(u^i)) = l_i
$$
and
$$
\mathbb L(\Sigma(v^j) - \Sigma(g^{(4)})) = \mathbb L(\Sigma(v^j) - \Sigma(g)) = l'_j
$$
for all $i,j$. Furthermore, for all $(x,t) \in \Delta \times L$, we find by using \eqref{06-05-21e}, \eqref{06-05-21d}, \eqref{06-05-21g} and \eqref{07-05-21b}, that

\begin{align*}
& t^{-l_i}( \Sigma(g^{(4)})(x,t) -\Sigma(u^i)(x,t)) \\
& = t^{-l_i}\left( \Sigma(g^{(3)})(x,t)  + \Sigma(g^{(2)})(x,t) -\Sigma(u^i)(x,t)\right)\\
& \geq t^{-l_i}(q(t)\Sigma(g^{(1)})(x,t) + (1-q(t))\Sigma(g)(x,t) - \Sigma(u^i)(x,t))  - 2\kappa  \\
& = t^{-l_i}q(t)(\Sigma(g^{(1)})(x,t) - \Sigma(u^i)(x,t)) \\
& \ \ \ \ \ \ \ \ \ \ \ \ \ \ \ \ \ \ + (1-q(t))\left(\Sigma(g)(x,t) - \Sigma(u^i)(x,t)\right) - 2\kappa \\
& \geq  t^{-l_i}\psi(t)(\Sigma(g^{(1)})(x,t) - \Sigma(u^i)(x,t)) \\
& \ \ \ \ \ \ \ \ \ \ \ \ \ \ \ \ \ \ + (1-q(t)) t^{-l_i}\left(\Sigma(g)(x,t) - \Sigma(u^i)(x,t)\right) - 3\kappa \  \\
& \geq  \psi(t)t^{-l_i}(\Sigma(g^{(1)})(x,t) - \Sigma(u^i)(x,t)) \\
& \ \ \ \ \ \ \ \ \ \ \ \ \ \ \ \ \ \ + (1-\psi(t)) t^{-l_i}\left(\Sigma(g)(x,t) - \Sigma(u^i)(x,t)\right) - 4\kappa \
\end{align*}
for all $(x,t) \in \Delta \times L$ and $i =1,2$. Now the properties of $g$ and $g^{(1)}$ stipulated in Observation \ref{05-05-21} and Observation \ref{06-05-21}, and the definition of $\psi$, imply that
\begin{align*}
& \psi(t)t^{-l_i}(\Sigma(g^{(1)})(x,t) - \Sigma(u^i)(x,t)) \\
& \ \ \ \ \ \ \ \ \ \ \ \ \ \ \ \ \ \ + (1-\psi(t)) t^{-l_i}\left(\Sigma(g)(x,t) - \Sigma(u^i)(x,t)\right) - 4\kappa \\
& \geq \min \{\epsilon,\epsilon_1\} - 4 \kappa \ 
\end{align*}
for all $(x,t) \in (F \times L) \cup (\Delta \times \{1\})$ and $i =1,2$. Essentially identical estimates show that 
$$
t^{-l'_j}\left( \Sigma(v^j)(x,t) - \Sigma(g^{(4)})(x,t)\right) \geq  \min \{\epsilon,\epsilon_1\} - 4 \kappa 
$$
for all $(x,t) \in (F \times L) \cup (\Delta \times \{1\})$ and $j =1,2$, and since $ \min \{\epsilon,\epsilon_1\} - 4 \kappa > 0$ we conclude that $u^i \leq g^{(4)} \leq v^j$ in $G$ for all $i,j$.
\end{proof}

Recall that an ideal $J$ in $(G,G^+)$ is a subgroup $J$ of $G$ such that $J = J \cap G^+ - J\cap G^+$ and $0 \leq h \leq g \in J \Rightarrow h \in J$. We shall not need to identify all ideals in $(G,G^+)$; it suffices here to establish the following:

\begin{lemma}\label{22-04-21d} Let $\rho$ be the automorphism \eqref{rho} and let $J$ be a non-zero ideal in $(G,G^+)$ such that $\rho(J) =J$. It follows that $J = G$.
\end{lemma}
\begin{proof} Let $0 \neq g \in J \cap G^+$. For $i \in \mathbb Z$ and $h \in H$ define $h^{(i)} \in G$ such that $(h^{(i)})_i = h$ and $(h^{(i)})_j = 0, \ j \neq i $. Let $h \in H^+ \backslash \{0 \}$.  Then $h^{(i)} \in G^+$ for all $i \in \mathbb Z$. Set $i' = \mathbb L(\Sigma(g))$. Since 
$$
t^{-i'} \Sigma(g)(x,t) \geq \epsilon > 0
$$
on $(F \times {L}) \cup (\Delta \times \{1\})$ for some $\epsilon > 0$, there is an $N \in \mathbb N$ such that
$$
Nt^{-i'} \Sigma(g)(x,t) - \theta(h)(x) \geq 1
$$
for all $(x,t) \in (F \times {L}) \cup (\Delta \times \{1\})$, which implies that 
$$
0 \leq h^{(i')} \leq Ng 
$$
in $G^+$ and hence that $h^{(i')} \in J$. Since $\rho(J) = J$ and $\rho({h}^{(i)}) = {h}^{(i-1)}$ for all $i$, it follows that ${h}^{(i)} \in J$ for all $i\in \mathbb Z$. The elements $h^{(i)}, h \in H^+, i \in \mathbb Z$, generate $G$ and we conclude therefore that $J = G$.
\end{proof}

\begin{lemma}\label{08-05-21} $(G,G^+)$ has large denominators.
\end{lemma}
\begin{proof} We adopt the notation from the proof of Lemma \ref{22-04-21d}. Let $g \in G^+\backslash \{0\}$ and $n \in \mathbb N$ be given. Set $l = \mathbb L(\Sigma(g))$. There is an $\epsilon > 0$ such that $t^{-l}\Sigma(g)(x,t) \geq \epsilon$ for all $(x,t) \in (F \times L) \cup (\Delta \times \{1\})$. Since $\theta(H)$ is dense in $\Aff(\Delta)$, there is an element $h \in H$ such that $0<\theta(h)(x) < \frac{\epsilon}{n}$ for all $x \in \Delta$. Then $0 \leq n h^{(l)} \leq g$ in $G$. On the other hand, there is an $m \in \mathbb N$ so big that
$$
t^{-l}\Sigma(g)(x,t) + 1 \leq m \theta(h)(x)
$$
for all $(x,t) \in \Delta \times L$ and then $g \leq mh^{(l)}$.  
\end{proof}

\subsubsection{Completing the proof when $K$ is unbounded}

From here the arguments are very similar, and many identical to those from the compact case.

It follows from Lemma \ref{21-04-21} and \cite{EHS} that there is an AF algebra $B$ whose $K_0$-group and dimension range is isomorphic to $(G ,G^+)$ and from \cite{E1} we conclude that $B$ is stable and that there is an automorphism $\gamma$ of $B$ such that $\gamma_* = \rho$ under the identification $K_0(B) = G$. By Lemma \ref{08-05-21} and Lemma \ref{28-10-20} we can choose $\gamma$ such that it has the same additional properties as listed in \ref{listrefx}. Set $C = B \rtimes_\gamma \mathbb Z$. No power of $\gamma$ is inner since $\gamma_*^k \neq \id_G$ for $k\neq 0$ and it follows from Lemma \ref{22-04-21d} that $B$ is $\gamma$-simple, and then from \cite{Ki1} that $C$ is simple. 

The traces $\tau_{y,t}$ and $\tau_x$ on $B$ can be defined by the same formulas as in Lemma \ref{18-12-20aaa}, and the following is an exact copy of Lemma \ref{18-12-20b}, but the different assumptions on $K$ and the different orderings of $G$ necessitate some slight changes to the proof.

\begin{lemma}\label{23-04-21} Let $\phi : G \to \mathbb R$ be a positive homomorphism such that $\phi([[u]])= 1$.  Assume that $\phi \circ \gamma_* = s^{-1} \phi$ for some $s \in ] 0,\infty[$. Then $s \in L$ and when $s \neq 1$ it follows that $\phi = {\tau_{y,s}}_*$ for some $y \in F$, and when $s =1$ it follows that $\phi = {\tau_x}_*$ for some $x \in \Delta$.
\end{lemma}
\begin{proof}
 In the following we denote by $C_0(\Delta \times L)$ and $C_0(L)$ the set of continuous real-valued functions vanishing at infinity on $\Delta \times L$ and $L$, respectively, and we denote by $\mathcal A(\Delta \times L)$ the elements $f$ of $C_0(\Delta \times L)$ with the property that $x \mapsto f(x,t)$ is affine for all $t \in L$; a closed subspace of $C_0(\Delta \times L)$. Set
$$
G' = \left\{ g\in G: \ \mathbb L(g) \leq -1 \right\} = \left\{ g \in G : \ g_n = 0 \ \forall n \geq 0 \right\} \ ;
$$
a subgroup of $G$. Then $\Sigma(G') \subseteq \mathcal A(\Delta \times L)$, and we claim that there is continuous linear map $\phi' : \mathcal A(\Delta \times {L}) \to \mathbb R$ such that $\phi = \phi' \circ \Sigma$ on $G'$. For this and later purposes we establish the following

\begin{obs}\label{07-05-21h} Let $h \in H$ and $r \in \mathbb R$ be such that $\theta(h)(x) < r$ for all $x \in \Delta$. It follows that $\phi([[h]]) < r$.
\end{obs} 
To prove this let $k\in \mathbb Z$ and $l \in \mathbb N$ be such that $\theta(h)(x) < \frac{k}{l} \leq r$ for all $x \in \Delta$. Then $l[[h]] \leq k [[u]]$ in $G$ and hence $l\phi([[h]]) \leq k$, from which the conclusion follows.

Assume $g \in G$ and that $\Sigma(g) = 0$. Let $N \in \mathbb N$. Since $\theta(H)$ is dense in $\Aff(\Delta)$ there is a $c \in H$ such that $0 < \theta(c)(x) < N^{-1}$ for all $x \in \Delta$ and Observation \ref{07-05-21h} implies that $0 \leq \theta([[c]]) \leq N^{-1}$. Since $\mathbb L(\Sigma(\pm g + [[c]])) = 0$ and $\Sigma(\pm g + [[c]])(x,t) \geq \epsilon > 0$, where $\epsilon = \min_{y \in \Delta} \theta(h)(y)$, it follows that $\pm g+[[c]]\in G^+$. We find in this way that $-\frac{1}{N} \leq \phi(g) \leq \frac{1}{N}$. Letting $N \to \infty$ we conclude that $\phi(g) =0$, and it follows that there is a homomorphism $\phi' : \Sigma(G') \to \mathbb R$ such that $\phi' \circ \Sigma = \phi$. Let $f \in \Sigma(G')$; say $f = \Sigma(g)$, where $g\in G'$, and let $k,l \in \mathbb N$ be natural numbers such that $|f(x,t)| < \frac{k}{l}$ for all $(x,t) \in \Delta \times {L}$. Since $\mathbb L\left(\Sigma(k[[u]] \pm lg)\right) = 0$ this implies that $k[[u]] \pm l g \in G^+$ and hence
$$
0  \ \leq  \ \phi(k[[u]] \pm l g) \ = \ k \pm l \phi'(f) \ .
$$
It follows that $\left|\phi'(f)\right| \leq \frac{k}{l}$, proving that $\phi'$ is Lipshitz continuous on $\Sigma(G')$. Since $\Sigma(G')$ is dense in $\mathcal A(\Delta \times {L})$ by Lemma \ref{06-05-21a} it follows that $\phi'$ extends by continuity to a linear map $\phi' : \mathcal A(\Delta \times {L}) \to \mathbb R$ such that $\phi'\circ \Sigma(g) = \phi(g)$ for $g \in G'$.

Let $T : \mathcal A(\Delta \times {L}) \to  \mathcal A(\Delta \times {L})$ be the operator
$$
T(\psi)(x,t) = t^{-1}\psi(x,t)  \ . 
$$ 
The equality 
\begin{equation}\label{18-20-12dd}
s^{-1}\phi' = \phi' \circ T \ 
\end{equation}
is established in the same way as in the proof of Lemma \ref{18-12-20b}. When $a \in \Aff(\Delta)$ and $f \in C_0({L})$ we denote by $a \otimes f$ the function $\Delta \times {L} \ni (x,s) \mapsto a(x)f(s)$. Assume $a(x)f(t) \geq 0$ for all $(x,t) \in (F \times L) \cup (\Delta \times \{1\})$. We claim that 
\begin{equation}\label{07-05-21e}
\phi'(a \otimes f) \geq 0 \ .
\end{equation} 
To establish this, let $\epsilon > 0$. We use Lemma \ref{22-04-21x} to get an element $g \in G'$ such
\begin{equation}\label{07-05-21g}
\left|a(x)f(t) - \Sigma(g)(x,t) \right| \leq \frac{\epsilon}{2}
\end{equation}
for all $(x,t) \in \Delta \times L$. Using the density of $\theta(H)$ in $\Aff(\Delta)$ we choose $h \in H^+$ such that $\epsilon \leq \Sigma(h)(x) < 2 \epsilon$ for all $x \in \Delta$. Then $\mathbb L(g+[[h]]) = 0$ and
$$
\Sigma(g+ [[h]])(x,t) \geq \frac{\epsilon}{2}
$$
for all $(x,t) \in (F \times L) \cup (\Delta \times \{1\})$, implying that $g+[[h]] \in G^+$ and hence that $\phi(g+[[h]]) = \phi'(\Sigma(g)) + \phi([[h]]) \geq 0$. It follows from Observation \ref{07-05-21h} that $\phi([[h]]) \leq 2\epsilon$, and we infer that $ \phi'(\Sigma(g)) \geq -2\epsilon$. Combined with \eqref{07-05-21g} it follows that $\phi'(a \otimes f) \geq -3\epsilon$, proving \eqref{07-05-21e}. In the same way as in the proof of Lemma \ref{18-12-20b} we deduce from \eqref{18-20-12dd} and \eqref{07-05-21e} that $s \in L$ and that there is a real number $\lambda(a)$ for each $a \in \Aff (\Delta)$ such that 
\begin{equation}\label{21-12-20dx}
\phi'(a \otimes f) \ = \ \lambda(a)f(s) 
\end{equation}
for all $f \in C_0(L)$. The resulting map $a \mapsto \lambda(a)$ is clearly linear and positive, and $\lambda(1) > 0$ since $\phi' \neq 0$. Then $a \mapsto \lambda(1)^{-1}\lambda(a)$ is a state of $\Aff(\Delta)$ and there is therefore an $x_0 \in \Delta$ such that $\lambda(a) = \lambda(1)a(x_0)$ for all $a \in \Aff(\Delta)$. Since the elements of $\left\{a \otimes f : \ a \in \Aff(\Delta) , \ f \in C_{0}({L}) \right\}$ span a dense set in $\mathcal A(\Delta \times {L})$ we find that
\begin{equation}\label{18-12-20ee}
\phi'(h) \ = \ \lambda(1)h(x_0,s) \ \ \ \ \forall h \in \mathcal A(\Delta \times {L}) \ .
\end{equation}
Note that $\gamma_*([[u]]) \in G'$ and $\Sigma(\gamma_*([[u]]))(x,t) = t^{-1}$. Using \eqref{18-12-20ee} and the assumptions on $\phi$ we find that
$$
\lambda(1)s^{-1} = \phi'(\Sigma(\gamma_*([[u]]))) = \phi(\gamma_*([[u]])) = s^{-1}\phi([[u]]) = s^{-1} \ ,
$$
implying that $\lambda(1) = 1$. Thus
$$
\phi (g) = \phi'(\Sigma(g)) = \sum_{m \in \mathbb Z} \theta(g_m)(x_0)s^m  
$$
when $g \in G'$. When $g \in G$ is a general element there is an an $n \in \mathbb N \cup \{0\}$ such that $g' = \gamma_*^n(g) \in G'$
and then
$$ 
\phi(g) = s^n \phi(g') =s^n\sum_{m \in \mathbb Z}\theta(g_{m+n})(x_0)s^{m} = \sum_{m \in \mathbb Z}\theta(g_m)(x_0)s^{m} \ .
$$
It remains only to show that $x_0 \in F$ when $s \neq 1$; which follows from Lemma \ref{21-12-20} in the same way as in the proof of Lemma \ref{18-12-20b}.
\end{proof}

 The rest of the proof of Theorem \ref{20-12-20c} for the unbounded case is identical to the proof from Section \ref{compact} for the compact case.

\section{Proof of Corollary \ref{15-12-20d} and Corollary \ref{27-10-20a}}\label{xxx}

We will combine Corollary \ref{20-12-20d} with the methods used in Appendix 12 of \cite{Th3} and Section 6 of \cite{Th4}. Given the UHF algebra $U$ in Corollary \ref{15-12-20d} we write
$$
U \simeq U_1 \otimes U_2 \ ,
$$
where $U_1$ and $U_2$ are both (infinite dimensional) UHF-algebras. Given the collection $\mathbb I$ of intervals and simplexes $S_I, I \in \mathbb I$, in Corollary \ref{15-12-20d} we get from Corollary 5.7 in \cite{Th4} a generalized gauge action $\alpha^1$ on $U_1$ with the following properties:
\begin{itemize}
\item For each $I \in \mathbb I$ and each $\beta \in I\backslash \{0\}$ there is a closed face $F_I$ in $S^{\alpha^1}_{\beta}$ which is strongly affinely isomorphic to $S_I$. 
\item For each $\beta \neq 0$ and each $\beta$-KMS state $\omega \in S^{\alpha^1}_{\beta}$ there is a unique norm-convergent decomposition
$$
\omega = \sum_{I \in \mathbb I_{\beta}} \omega_I \ 
$$
where $\omega_I \in \mathbb R^+F_I$. 
\end{itemize}
That $\alpha^1$ is a generalized gauge action means that there is a Bratteli diagram $\Br$ and a map $F : \Br_{Ar} \to \mathbb R$ defined on the set $\Br_{Ar}$ of arrows in $\Br$ which define $\alpha^1$ in the following way. Let $\mathcal P_n$ denote the set of paths of length $n$ in $\Br$ starting at the top vertex. Extend the map $F$ to $\mathcal P_n$ such that
$$
F(\mu) = \sum_{i=1}^n F(a_i) \ ,
$$
when $\mu = a_1a_2\cdots a_n$ is made up of the arrows $a_i$. The Bratteli diagram $\Br$ is chosen such that  there are finite dimensional $C^*$-subalgebras $\mathbb F_n$ in $U_1$ spanned by a set of matrix units $E^n_{\mu,\mu'}, \ \mu,\mu' \in \mathcal P_n$, such that $\mathbb F_n \subseteq \mathbb F_{n+1}$ for all $n$ and
$$
U_1 = \overline{\bigcup_n \mathbb F_n} \ . \
$$
The flow $\alpha^1$ is defined such that
$$
\alpha^1_t\left( E^n_{\mu,\mu'}\right) = e^{i (F(\mu)-F(\mu'))t} E^n_{\mu,\mu'} \ .
$$
See \cite{Th4}. In order to ensure that $\alpha^1$ is $2\pi$-periodic it is necessary to arrange that $F$ only takes integer values. That this is possible follows by inspection of the proof in \cite{Th4}; in fact, the only step in the proof where this is not automatic is in the proof of Lemma 5.3 in \cite{Th4}, where some real numbers $t_k$ are chosen. The only crucial property of these numbers is that they must be sufficiently big and they may therefore be chosen to be natural numbers. The resulting potential $F$ will then be integer-valued. Let $\alpha^2$ be a flow on $U_2$ with the properties specified for $\alpha$ in Corollary \ref{20-12-20d} and set
$$
\alpha_t = \alpha^1_t \otimes \alpha^2_t \ ,
$$
which we will argue defines a flow on $U_1 \otimes U_2$ with the properties stated in Corollary \ref{15-12-20d}. Since a $\beta$-KMS state for $\alpha$ will restrict to a $\beta$-KMS state for $\alpha^2$ on the tensor factor $U_2$, it follows from  Corollary \ref{20-12-20d} that there are no $\beta$-KMS states for $\alpha$ unless $\beta \in K$. Let $\beta \in K$ and let $\omega_\beta$ be the unique $\beta$-KMS state for $\alpha^2$. It remains only to show that the map $\omega \ \mapsto \ \omega \otimes \omega_\beta$ is an affine homeomorphism from $S^{\alpha^1}_\beta$ onto $S^{\alpha}_\beta$. Note that only the surjectivity is not obvious, and only when $\beta \neq 0$. Consider therefore a $\beta$-KMS state $\psi \in S^{\alpha}_\beta$, $\beta \neq 0$. Let $x,y \in U_1, \ b \in  U_2$ such that $x$ is $\alpha^1$-analytic. Then $x \otimes 1$ is $\alpha$-analytic and $\alpha_{i \beta}(x \otimes 1) = \alpha^1_{i\beta} (x) \otimes 1$. Hence
\begin{align*}
&\psi(xy \otimes b) = \psi((x \otimes 1)(y \otimes b)) = \psi((y\otimes b) \alpha_{i \beta}(x\otimes 1)) \\
& = \psi((y\otimes b) (\alpha^1_{i \beta}(x) \otimes 1)) = \psi(y\alpha^1_{i \beta}(x) \otimes b) \ .
\end{align*}
Thus, if $b \geq 0$ and $b \neq 0$, the map $U \ni x \mapsto \psi(x \otimes b)$ is a $\beta$-KMS functional for $\alpha^1$, and hence
\begin{align*}
&\psi(E^n_{\mu, \mu'} \otimes b) = \psi( \alpha^1_{i\beta}(E^n_{\mu, \mu'})\otimes b)) =  e^{-\beta (F(\mu) - F(\mu'))} \psi(E^n_{\mu, \mu'} \otimes b) \ .
\end{align*}
Since $\beta \neq 0$ it follows that $\psi(E^n_{\mu,\mu'} \otimes b) \ = \  0$ when $F(\mu) \neq F( \mu')$, and hence that $\psi$ factorises through the map $Q \otimes \id_{U_2}$, where $Q : U_1 \to U_1^{\alpha^1}$ is the conditional expectation onto the fixed point algebra $U_1^{\alpha^1}$ of $\alpha^1$. Let $a \geq 0$ in $U_1$. Then
$$
U_2 \ni b \ \mapsto \ \psi(a \otimes b) = \psi(Q(a) \otimes b)
$$
is a $\beta$-KMS functional for $\alpha^2$ since $\psi$ is a $\beta$-KMS state for $\alpha$ and $Q(a)$ is fixed by $\alpha^1$. The uniqueness of $\omega_\beta$ implies therefore that
$$
\psi(a \otimes b) = \lambda(a)\omega_{\beta}(b)
$$
for some $\lambda(a) \in \mathbb R$ and all $b \in U_2$. It is then straightforward to show that $\lambda(a)$ can be defined for all $a \in U_1$, resulting in a $\beta$-KMS state $\lambda$ for $\alpha^1$. This completes the proof of Corollary \ref{15-12-20d}. 

 Corollary \ref{27-10-20a} follows by choosing the family $\mathbb I$ of intervals in Corollary \ref{15-12-20d} in an appropriate way, e.g. as the collection of all bounded intervals with rational endpoint, and also the Choquet simplexes $S_I$ in an appropriate way. See Section 6 in \cite{Th4}.

\end{document}